\documentclass[a4paper, 10pt, twoside, notitlepage]{amsart}

\usepackage[utf8]{inputenc}
\usepackage{color}
\usepackage{amsmath}
\usepackage{amssymb}
\usepackage{amsthm}
\usepackage{mathrsfs}
\usepackage{mathdots}
\usepackage{geometry}
\usepackage{graphicx}
\graphicspath{{figures/}}
\usepackage{esint}
\usepackage[colorlinks=true,linkcolor=blue]{hyperref}
\usepackage{comment}
\usepackage{combelow}

\usepackage[capitalise, noabbrev, nameinlink]{cleveref}
\crefname{equation}{}{}
\crefname{enumi}{}{}
\crefname{lem}{Lem\-ma}{Lemmata}
\crefname{section}{Section}{Sections}
\crefname{subsection}{Section}{Sections}
\usepackage{tikz, tikz-3dplot, pgfplots} \usetikzlibrary{patterns} \usetikzlibrary{calc}

\theoremstyle{plain}
\newtheorem{thm}{Theorem}
\newtheorem{prop}{Proposition}[section]

\newtheorem{lem}[prop]{Lemma}
\newtheorem{cor}[prop]{Corollary}

\newtheorem{defi}[prop]{Definition}
\newtheorem{rmk}[prop]{Remark}

\newtheorem{example}[prop]{Example}
\newtheorem{con}[prop]{Convention}

\newcommand{\R}{\mathbb{R}}

\newcommand{\N}{\mathbb{N}}

\newcommand{\p}{\partial}

\newcommand{\diam}{\textup{diam}}

\newcommand{\A}{\mathcal{A}}
\newcommand{\B}{\mathcal{B}}
\renewcommand{\AA}{\mathbb{A}}

\renewcommand{\S}[1]{\textup{Sym}(\R^d;#1)}
\newcommand{\Ds}{D^{\text{sym}}}

\DeclareMathOperator{\di}{div}
\DeclareMathOperator{\curl}{curl}
\DeclareMathOperator{\dist}{dist}
\DeclareMathOperator{\Per}{Per}
\DeclareMathOperator{\vspan}{span}

\title[Scaling for higher order two-well problems]{On Scaling Properties for a Class of Two-Well Problems for Higher Order Homogeneous Linear Differential Operators}

\author{Bogdan Rai\cb{t}\u a} \address{Department of Mathematics and Statistics, Georgetown University, 3700 O St NW, Washington, DC 20057, United States of America} \email{bogdanraita@gmail.com}

\author{Angkana Rüland} \address{Institute for Applied Mathematics and
Hausdorff Center for Mathematics, University of Bonn, Endenicher Allee
60, 53115 Bonn, Germany} \email{rueland@uni-bonn.de}

\author{Camillo Tissot} \address{Institute for Applied Mathematics,
University of Bonn, Endenicher Allee 60, 53115 Bonn, Germany}
\email{camillo.tissot@uni-bonn.de}

\author{Antonio Tribuzio} \address{Institute for Applied Mathematics,
University of Bonn, Endenicher Allee 60, 53115 Bonn, Germany}
\email{tribuzio@iam.uni-bonn.de}

\begin{document}
	
\begin{abstract}
  We study the scaling behaviour of a class of compatible two-well problems for higher order, homogeneous linear differential operators. To this end, we first deduce general lower scaling
  bounds which are determined by the vanishing order of the symbol of the operator on the unit sphere in direction of the associated element in the wave cone. We complement the lower bound estimates by a
  detailed analysis of the two-well problem for generalized (tensor-valued) symmetrized derivatives with the help of the (tensor-valued) Saint-Venant compatibility conditions. In two spatial dimensions for highly symmetric boundary data (but arbitrary tensor order $m \in \mathbb{N}$) we provide upper bound constructions matching the lower bound estimates. This illustrates that for the two-well problem for higher order operators new scaling laws emerge which are determined by the Fourier symbol in the direction of the wave cone. The scaling for the symmetrized gradient from \cite{CC15} which was also discussed in \cite{RRT23} provides an example of this family of new scaling laws.
\end{abstract}
	
	\maketitle
	
\section{Introduction}
	\label{sec:intro}
	It is the objective of this article to quantitatively study the two-well problem for a class of \emph{higher order}, constant coefficient, linear differential operators generalizing the
	curl and curl curl as the annihilators of the gradient and of the symmetrized gradient, respectively. We seek to illustrate that for this class of operators and a suitable class of wells the
	\emph{maximal vanishing order} of the associated symbols on the unit sphere determines the scaling behaviour of corresponding singular perturbation problems. In particular, we show that such higher order operators may lead to a
	scaling behaviour which does no longer satisfy the typical $\epsilon^{\frac{2}{3}}$ scaling behaviour, first obtained in \cite{KM92,KM94}.
	
	\subsection{The classical two well-problem for the gradient}
	In order to put our results into perspective, let us first recall the classical (compatible) two-well problem for the gradient: The qualitative and quantitative compatible two-well problem for
	the gradient without gauge invariance is a well-studied problem motivated by questions from materials science \cite{BJ92,M1, B}. In its quantitative forms it is a prototypical problem in
	the vector-valued calculus of variations, giving rise to pattern formation problems.
	
	As a first \emph{qualitative} observation, one notes that the only exact solutions to the differential inclusion corresponding to the compatible two-well problem
	\begin{align}
		\label{eq:two_grad}
		\begin{split}
			\nabla v &\in \{A,B\} \mbox{ a.e. in } \Omega, \ v\in W^{1,\infty}_{loc}(\R^d; \R^d),
		\end{split}
	\end{align}
	with $A,B\in \R^{d \times d}$ and $\text{rank}(A-B)=1$, are so-called \emph{simple laminates}
	\cite{BJ92, M1}. These  solutions are locally of the form $v(x)= f(n\cdot x) + \text{affine function}$, where $f: \R \rightarrow \R^d$ and $n\in \mathbb{S}^{d-1}$ is (up to its sign) determined by the relation $A-B = a\otimes n$ for
	some $a\in \R^d$. In particular, for $\lambda \in (0,1)$ and $F_{\lambda}:= \lambda A + (1-\lambda) B$, there are no solutions to \cref{eq:two_grad} with $\nabla v = F_{\lambda}$ in $\R^d\setminus\Omega$ if $\Omega$ is, e.g., bounded.
	
	When viewing the two-well problem energetically, by minimizing elastic energies of the form
	\begin{align*}
		E_{el}( v, \chi):= \int\limits_{\Omega} |\nabla v - \chi_A A - \chi_B B|^2 dx,
	\end{align*}
	where for every $x\in \Omega$ we let $\chi(x) := \chi_A(x) A + \chi_B(x) B \in \{A,B\}$, $\chi_A(x), \chi_B(x)\in\{0,1\}$, among
	\begin{align*}
		\mathcal{D}_{F_{\lambda}} :=\{v \in W^{1,2}_{loc}(\R^d; \R^d): \ \nabla v = F_{\lambda} \mbox{ in } \R^d \setminus \overline{\Omega}\},
	\end{align*}
	a rather different behaviour emerges: Although no exact solutions to \cref{eq:two_grad} exist, due to the lack of lower-semi-continuity it still holds that
	$\inf_{\chi \in L^2(\Omega;\{A,B\})}\inf_{v \in \mathcal{D}_{F_\lambda}}E_{el}(v,\chi)\allowbreak=0$. The boundary conditions enforce oscillations and thus result in infinitely fine-scale structure \cite{BJ92,M1,B}. A
	relaxation leads to the notion of gradient Young measure solutions, a type of generalized solutions (parametrized measures) which describe the oscillatory behaviour of minimizing sequences.
	
	Motivated by the discrepancy between the exact differential inclusion and its energetically quantified version and seeking to study finer properties of the two-well problem, an important  class of models consists of \emph{singular perturbation models}, penalizing fine oscillations. Instead of only minimizing the elastic energy, one here also considers additional (regularizing) surface energies
	\begin{align*}
		E_{surf}(\chi):= \int\limits_{\Omega} |\nabla \chi|,
	\end{align*} 
	and
	\begin{align*}
		E_{\epsilon}(v,\chi) := E_{el}( v, \chi) + \epsilon E_{surf}(\chi).
	\end{align*}
	Here $\int_{\Omega} |\nabla \chi|$ denotes the total variation norm of $\nabla \chi$, the distributional gradient of $\chi \in BV(\Omega;$ $ \{A,B\})$.  Due to the higher order regularization term, the energy does no longer permit arbitrarily
	fine oscillations but introduces a length scale depending on $\epsilon>0$ and thus selects microstructure, e.g., it can distinguish between simple laminate and branching type
	structures \cite{KM92,KM94}. Comparing the regularized singular perturbation problem and the non-regularized ``elastic'' energies, it is particularly interesting to investigate the limit
	$\epsilon \rightarrow 0$ and the scaling behaviour of the singular perturbation problem in $\epsilon>0$ as $\epsilon \rightarrow 0$. In this context, the celebrated results \cite{KM92,KM94}
	assert that the minimal energy does \emph{not} display the scaling behaviour of simple laminates but that of branching type structures (at least in generic, non-degenerate domain geometries),
	see also \cite{CC15}. In \cite{C1} this observation is further strengthened by proving that minimizers are asymptotically self-similar.  The scaling behaviour of these energies thus encodes
	important information on the interaction of the differential constraint (i.e., the condition of dealing with gradients) and the nonlinearity (i.e., the two-well nature of the problem).
	Motivated by problems from materials science, similar $\epsilon^{\frac{2}{3}}$ results have been obtained for generalizations of the differential constraint, including, for instance,
	divergence and symmetrized gradient constraints \cite{CO09, CO12, KW14,CKO99}. However, it was pointed out in the work \cite{CC15} that the scaling behaviour does not \emph{always} have to be
	of the order $\epsilon^{\frac{2}{3}}$: Indeed, in \cite{CC15} in a ``degenerate'' setting (with only one rank-one direction, compared to the generic setting of two rank-one directions) a scaling of the order $\epsilon^{\frac{4}{5}}$ was observed.

	\subsection{The compatible two-well problem for constant coefficient, linear differential operators}
	Motivated by the outlined problems from materials science and the qualitative study of the incompatible two-well problem from \cite{DPPR18} for homogeneous linear differential operators,
	in \cite{RRT23}, we started to systematically study the scaling properties of the compatible two-well problem depending on the class of differential operators at
	hand. To this end, we considered general, homogeneous, constant coefficient, linear differential operators
	\begin{align*}
		\A(D)u:= \sum\limits_{|\alpha|=m} A_{\alpha} \partial^\alpha u,
	\end{align*}
	for $u: \R^d \rightarrow \R^n$, $A_{\alpha} \in \R^{k \times n}$.  The two-well problem from \cref{eq:two_grad} then turns into:  Find $u \in
		L^2_{loc}(\R^d;\R^n)$ such that
	\begin{align}
		\label{eq:two_well_gen}
		\begin{split}
			u& \in \{A,B\} \mbox{ in } \Omega,\\
			\A(D) u &= 0 \mbox{ in } \R^d \ \text{distributionally},\\
			u& = F_{\lambda} \mbox{ in } \R^d \setminus \overline{\Omega},
		\end{split}
	\end{align}
	where $A-B \in \Lambda_{\A} \setminus I_{\A}$ and $F_{\lambda} = \lambda A + (1-\lambda) B$ with $\lambda \in (0,1)$. Here
	\begin{align} \label{eq:wave_cone}
		\Lambda_{\A}:= \bigcup\limits_{\xi \in \mathbb{S}^{d-1}}\ker \AA(\xi),
	\end{align}
	denotes the \emph{wave cone} introduced in \cite{T79,Mur81} and
	\begin{align} \label{eq:super_compatible}
		I_{\A}:= \bigcap\limits_{\xi \in \mathbb{S}^{d-1}} \ker \AA(\xi),
	\end{align}
	is the set of \emph{supercompatible states} introduced in this context in \cite{RRT23}, where we denote the \emph{symbol} of $\A(D)$ by
		\begin{align} \label{eq:symbol}
			\AA(\xi) := \sum_{|\alpha|=m} A_\alpha \xi^\alpha.
                \end{align}
              It is known that the wave cone generalizes the presence of rank-one connections for the curl operator and that it is possible to construct generalized simple laminates from it.
              Indeed, for $h:\R \to \{0,1\}$, $A-B \in \Lambda_{\A}$ and $\xi \in \R^d$ such that $A-B\in \ker \AA(\xi)$, the function
	\begin{align*}
		u(x):= (A-B) h(\xi \cdot x) + B
	\end{align*}
	is a one-dimensional solution to $u \in \{A,B\}$, $\A(D) u = 0$.
	
	Similarly as in the case of the gradient inclusion, it is possible to associate a singular perturbation problem to \cref{eq:two_well_gen}. To this end, we consider
	\begin{align}
		\label{eq:energy_total_gen}
		E_{\epsilon}^{\A}(\chi;F_\lambda):= E_{el}^{\A}(\chi; F_{\lambda}) + \epsilon E_{surf}^{\A}(\chi) :=  \inf\limits_{u\in \mathcal{D}_{F_{\lambda}}^{\A}} \int\limits_{\Omega}|u-\chi|^2 dx+ \epsilon \int\limits_{\Omega} |\nabla \chi|,
	\end{align}
	where $\chi \in BV(\Omega; \{A,B\})$ and
	\begin{align}
		\label{eq:admissible_gen}
		\mathcal{D}_{F_{\lambda}}^{\A} := \{u \in L^2_{loc}(\R^d;\R^n): \  \A(D)u = 0  \mbox{ in } \mathcal{D}'(\R^d), \ u= F_{\lambda} \mbox{ in } \R^d \setminus \overline{\Omega}\}.
	\end{align}
	In the following we will often omit the superscript $\A$ in the notation of the energy $E_\epsilon^\A$ and the set $\mathcal{D}_{F_\lambda}^{\A}$.
	We remark that the outlined setting easily generalizes to operators acting on fields $u:\R^d\to V$ with $V$ being a (real) vector space of dimension $n$.
	In particular, in what follows we will consider the cases $V=\S{m}$ and $V=\R^k\otimes\S{m}$ (see \cref{sec:Saint_venant}).

	Within this setting, in the article \cite{RRT23}, as one of the main results, the first three authors proved that for \emph{first order}, constant coefficient differential operators, the lower
	$\epsilon^{\frac{2}{3}}$ scaling behaviour is generic, provided that the wells are chosen to be \emph{compatible} but not supercompatible.

	\begin{thm}[{\cite[Theorem 1]{RRT23}}] \label{thm:FirstOrderScaling}
          Let $d,n \in \N$. Let $\Omega \subset \R^d$ be open, bounded and Lipschitz.  Let $\A(D)$ be a homogeneous, constant coefficient, linear, \emph{first order} differential operator
          and $A,B \in \R^n$ such that $A-B \in \Lambda_{\A} \setminus I_{\A}$, see \cref{eq:wave_cone} and \cref{eq:super_compatible}. Let $\lambda \in (0,1)$,
            $F_\lambda = \lambda A + (1-\lambda)B$ and let $E^{\A}_{\epsilon}$ be as in \cref{eq:energy_total_gen} above. Then, there exist constants $C=C(\A(D),\Omega,d,\lambda, A, B)>0$ and
          $\epsilon_0 = \epsilon_0(\A(D),\Omega,d,\lambda,A,B)>0$ such that for any $\epsilon \in (0,\epsilon_0)$
		\begin{align*}
			\inf\limits_{\chi \in BV(\Omega;\{A,B\})}  E_{\epsilon}^{\A}(\chi;F_\lambda) \geq C \epsilon^{\frac{2}{3}}.
		\end{align*}
	\end{thm}

	While our result illustrates that the lower bounds from the two-well problem for the gradient persist for general \emph{first order} operators, our argument strongly relied on the \emph{linearity} of the associated first order symbols and the linear structure of the null set of
	\begin{align}\label{eq:symbol(A-B)}
		\xi\in\R^{d}\mapsto \AA(\xi)(A-B).
	\end{align}
	 Revisiting the example from \cite{CC15} (which involves a second order differential operator) we could prove
	that the vanishing order of \cref{eq:symbol(A-B)} gives rise to a \emph{different}, non-$\epsilon^{\frac{2}{3}}$-scaling behaviour which had first been identified in \cite{CC15}.  It however remained
	an open problem to investigate the behaviour of \emph{more general, higher order operators}.  In this context, central questions are:
	\begin{itemize}
		\item[(Q1)] What scaling behaviour can emerge for higher order, homogeneous linear operators for their associated singularly perturbed compatible two-well problems?
		\item[(Q2)] Can one identify structural conditions (for $\A(D)$ and $A-B \in \Lambda_{\A}\setminus I_{\A}$) giving rise to certain scaling behaviour?
	\end{itemize}
	In the present article, we begin to systematically investigate these questions by considering a rather general class of homogeneous, linear differential operators and by deducing lower bound
	estimates for these.  These lead to non-standard, non-$\epsilon^{\frac{2}{3}}$-lower bound scaling behaviour which is directly associated with the structure of the symbol \cref{eq:symbol(A-B)}. For a particular class of operators generalizing the symmetrized gradient and a class of particular wells and boundary data we complement these
	lower bounds with matching upper bounds proving their optimality.

	\subsection{Outline of the main results}
	\label{sec:main}
	Continuing our investigation from \cite{RRT23}, it is our objective to investigate the scaling properties of \emph{higher order compatible two-well problems}. To this end, on the one hand, we
	study \emph{specific families} of such operators in detail. On the other hand, we systematically investigate the lower bound scaling behaviour for a rather general class of constant coefficient,
	homogeneous, linear differential operators.  Our results and in particular the specific example classes illustrate that, in general, for higher order, homogeneous, linear differential
	operators \emph{different scaling behaviour} may arise than for first order operators. More precisely, we show that the scaling behaviour for our classes of operators depends in a precise way
	on the \emph{maximal vanishing order} of the associated symbol restricted to the unit sphere. This generalizes and systematizes the example from \cite[Section 3.5]{RRT23} which was based on the results from \cite{CC15} and provides
	further, new scaling laws of higher order.
	
	In order to explain this, let us describe the precise set-up of our problem.  In what follows, we will first focus on a family of homogeneous, higher-order linear differential operators,
	generalizing the curl (and curl curl) operators as well as their associated potentials. 
	In a second step, we will then discuss a rather general class of lower scaling bounds for which the estimates will be determined by the \emph{maximal vanishing order} of the associated symbol \cref{eq:symbol(A-B)} restricted to the unit sphere.

	\subsection{Scaling results for generalized symmetrized gradients}		
	Let us begin by considering the case of the curl (and curl curl) operator and its generalizations. It is the content of Poincar\'e's Lemma that for a simply connected domain a \emph{one-tensor
		field} (i.e., a vector field) is a \emph{gradient}, if and only if its curl vanishes. Similarly, it is well-known and often used in geometrically linearized elasticity that a \emph{symmetric two-tensor
		field} (i.e., a symmetric matrix field) on a simply connected domain is a \emph{symmetrized gradient}, if and only if it satisfies the Saint-Venant compatibility conditions. More precisely, on the
	whole space, a necessary and sufficient condition for a tensor field $u: \R^{d} \to \R^{d \times d}_{\text{sym}}$ to be a symmetrized derivative associated with some $v: \R^d \to \R^d$, i.e.,
	\begin{align*}
		u=e(v):= \Ds v := \frac{1}{2}(\nabla v + (\nabla v)^t),
	\end{align*}
	is given by the vanishing of the $\curl \curl$ operator, or equivalently the validity of the system of differential equations given by
	\begin{align*}
		\p_{kl}^2 u_{ij} + \p_{ij}^2 u_{kl} - \p_{jk}^2 u_{il} - \p_{il}^2 u_{jk} = 0, \ i,j,k,l \in \{1,\dots,d\}.
	\end{align*}
	
	\emph{Higher order symmetric tensors} play a major role in inverse problems and tensor tomography. In the Euclidean setting, the ray transform is given by
	\begin{align*}
          & I_m: C_c^{\infty}(\R^d, \S{m}) \rightarrow C^{\infty}(T \mathbb{S}^{d-1}), \\
          & I_m f(x,\xi):=\int_{\R} \sum_{i_1,\dots,i_m=1}^d  f_{i_1\dotsi i_m}(x + t \xi) \xi^{i_1}\dots \xi^{i_m}  dt.
	\end{align*}
	Seeking to recover the higher rank tensor $f$ from measurements of $I_m f$, a (generalized) Helmholtz type decomposition into a potential and solenoidal part plays a major role (see, for instance, \cite{PSU14,IM19,PSU23, Sha94} dealing with geometric versions of the ray transform). Indeed, it is at best the solenoidal part of $f$ which can be reconstructed from the knowledge of $I_m$, while the potential part is characterized by the generalized Saint-Venant compatibility conditions and cannot be inferred from the measurements of $I_m f$. Hence, additional structural conditions, such as the two-valuedness, of the potential part may be imposed and investigated.
	
	Following \cite[Chapter 2]{Sha94}, in
	this article as a model class of higher order differential operators we will study such generalizations of the Saint-Venant compatibility condition for \emph{higher rank tensors fields}
	\begin{align*}
		u : \R^d \to \S{m}:=\{M \in (\R^d)^{\otimes m}: M \mbox{ is symmetric}\}
	\end{align*}
	and their interaction with the nonlinear constraint given by the two-well problem.
	More precisely, for $u : \R^d \to \S{m}$, we consider the $m$-th order differential operator $\A(D)$ with
	\begin{align}
		\label{eq:compat_op}
		[\A(D)u]_{i_1 j_1 i_2 j_2 \dots i_m j_m} := \alpha_{i_1 j_1} \circ \alpha_{i_2 j_2} \circ \dots \circ \alpha_{i_m j_m} \Big( \p_{j_1 \dots j_m}^m u_{i_1 \dots i_m} \Big).
	\end{align}
	Here, by $\alpha_{i_k i_l}$ we denote the \emph{anti-symmetrization operator} in the indices $i_k, i_l$ by
	\begin{align*}
		\alpha_{i_k i_l}(M_{i_1 \dots i_k \dots i_l \dots i_{2m}}) := \frac{1}{2} \Big( M_{i_1 \dots i_{2m}} - M_{i_1 \dots i_l \dots i_k \dots i_{2m}} \Big).
	\end{align*}
        Now the differential operator $\mathcal{A}(D)$ generalizes the Saint-Venant compatibility conditions in the sense that it holds (for compactly supported maps) that $\A(D)u = 0$ for
	$u: \R^d \to \S{m}$ if and only if $u$ is the symmetrized derivative of a tensor valued map \cite[Theorem 2.2.1]{Sha94}. In this context, we say that a mapping $u: \R^d \to \S{m}$ is a
	\emph{symmetrized derivative}, if there is $v: \R^d \to \S{m-1}$ such that $u = \Ds v$, cf. \cref{eq:potential} in \cref{sec:tensors}.
	We consider some specific examples of this set-up in \cref{ex:d-and-m} in \cref{sec:Saint_venant}.

      Generalizing the $\epsilon^{\frac{2}{3}}$-scaling result for the gradient inclusion from \cite{KM92,KM94} and the observations made in \cite{CC15} that for $\A(D) = \curl \curl$ there are instances of an $\epsilon^{\frac{4}{5}}$- scaling, we prove the following analogous result to \cref{thm:FirstOrderScaling} for this model class of operators.
        
	\begin{thm}[Symmetrized derivative]
		\label{thm:scaling_2D_new}
		Let $d,m \in \N, d \geq 2$ and $l \in \N^d$. Let $\Omega \subset \R^d$ be an open, bounded Lipschitz domain. Let $E_{\epsilon}^{\A}(\chi;F)$ be as above in \cref{eq:energy_total_gen}  with the operator $\A(D)$ given in \cref{eq:compat_op}. Then the following scaling results hold:
		\begin{itemize}
			\item \underline{Sharp scaling bounds for $d=2$, $\lambda=\frac{1}{2}$}. Let $d=2$, $A-B = e_1^{\odot l_1} \odot e_2^{\odot l_2}$ (see \cref{eq:sym_prod} for the symmetric
			tensor notation) such that $l_1 + l_2 =m$ and $F= \frac{1}{2}A+\frac{1}{2}B$.  Moreover let $\Omega = (0,1)^2$. Then there exist constants
			$C>1$ and $\epsilon_0 > 0$ (depending on $m$) such that for $L:= \max\{l_1, l_2\}$ and for any $\epsilon \in (0,\epsilon_0)$
			\begin{align*}
				C^{-1} \epsilon^{\frac{2L}{2L+1}} \leq \inf\limits_{\chi \in BV(\Omega;\{A,B\})} E_{\epsilon}^{\A}(\chi;F) \leq C \epsilon^{\frac{2L}{2L+1}}.
			\end{align*}
			\item \underline{Lower scaling bounds for $d\geq 2$}. Let $d\geq 2$, $A-B = e_1^{\odot l_1} \odot e_2^{\odot l_2} \odot \cdots \odot e_d^{\odot l_d}$ such
			that $\sum_{j = 1}^{d} l_{j} = m$ and $F_{\lambda}:= \lambda A + (1-\lambda ) B$ for some $\lambda \in (0,1)$. Then there exist $C>0$ and $\epsilon_0>0$ (depending on $d, m, \Omega$ and $\epsilon_0$ also depending on $\lambda$) such
			that for $L:=\max\limits_{j\in\{1,2,\dots,d\}}l_j$ and for any $\epsilon \in (0,\epsilon_0)$
			\begin{align*}
				C \min\{1-\lambda,\lambda\}^2\epsilon^{\frac{2L}{2L+1}} \leq \inf\limits_{\chi \in BV(\Omega;\{A,B\})} E_{\epsilon}^{\A}(\chi; F_{\lambda}) .
			\end{align*}
		\end{itemize}
	\end{thm}
	
	Let us comment on the assertions of the theorem: Firstly, we remark that the scaling laws in \cref{thm:scaling_2D_new} represent a \emph{new} class of scaling laws for the two-well
	problem for higher order, linear differential operators. Secondly, this class of operators systematizes the observations from \cite{CC15} and \cite[Section 3.4]{RRT23} in which a
	non-$\epsilon^{\frac{2}{3}}$-scaling behaviour emerges and which is determined by the \emph{vanishing order of the associated symbols} on the unit sphere.
        We also mention that for this class of differential operators modelled on generalizations of the gradient and symmetrized gradient, our results from \cref{thm:scaling_2D_new} do not
	cover all possible settings of $A,B $ such that $A-B \in \Lambda_{\A}$. Indeed, our choices of $A-B$ form a basis of the wave cone and yield the scaling behaviour on these basis
	vectors. This however \emph{does not} yield the scaling behaviour of a general element in the wave cone. We seek to return to this in future work. 
	Finally, let us comment on the constraints in the upper bound constructions in \cref{thm:scaling_2D_new}. The condition that $\lambda = \frac{1}{2}$ provides strong symmetry properties. In particular, it allows for both odd and even reflections of certain building block constructions (see \cref{lem:UnitCell}). For general $\lambda \in (0,1)\setminus \{\frac{1}{2}\}$ only weaker replacements (of odd reflections) are available which do not allow for an immediate generalization to an arbitrary choice of $\lambda \in (0,1)$ and general tensor order. We further remark that providing matching upper bound constructions in higher dimensions $d>2$ leads to new technical difficulties even in the case of the model operators. Indeed, in this case one would need to ensure the validity of the prescribed Dirichlet data on \emph{all} sides of the cube. The construction given in this article only achieves the boundary datum on four sides, and would thus have to be modified correspondingly to achieve an admissible deformation in higher dimensions. To account for this, ``rotation-type'' arguments have been introduced in \cite{RRT23,RTrib22} to match the Dirichlet data on all faces. It is expected that -- at the expense of additional technicalities -- similar ideas could also be of relevance in our context for the model operators under consideration. Finding matching upper bound constructions under the given strong Dirichlet conditions for general operators \emph{beyond} model settings however remains a major challenge.

        We show in \cref{lem:LowerScalingDiv} that a similar result as stated in \cref{thm:scaling_2D_new} holds for the $m$-th order divergence, which is defined in \cref{eq:DivergenceOp}.

	\subsection{Lower bounds for a class of linear, homogeneous differential operators}
	Building on the specific example of the higher order curl and its potential, the generalized symmetrized derivative, we seek to study the scaling behaviour of more general two-well problems for homogeneous, constant coefficient,
	linear differential operators. To this end, we systematically deduce lower scaling bounds for a rather large class of linear operators. As in the previous section, we consider singular perturbation problems as in \cref{eq:energy_total_gen},
	\cref{eq:admissible_gen}. A crucial role to determine lower bounds is played by the polynomial $p(\xi) = |\AA(\xi)(A-B)|^2$, cf. \cite[Corollary 3.2]{RRT23}.	
	To that end, we introduce the \emph{maximal vanishing order on the unit sphere} of the non-negative, homogeneous polynomial $p:\R^d \to \R$.
	
	\begin{defi}[Maximal vanishing order on the unit sphere]
		\label{defi:index}
		Let $p \in \R[\xi]$ be a non-negative, 2m homogeneous polynomial. Let $V$ denote the zero set of $p$. We then define the \emph{maximal vanishing order} $L[p]$ of $p$ as
		\begin{align*}
			L[p]:= \min\left\{\ell \in \N: \ \inf\limits_{\xi \in \mathbb{S}^{d-1} \setminus V} \frac{p(\xi)}{\dist_V(\xi)^{2\ell}}>0 \right\},
		\end{align*}
                where we denote the distance function to $V$ by
                  \begin{align}
                  \label{eq:dist}
                    \dist_V(\xi) := \inf \{|\xi- \zeta|: \zeta \in V\}.
                  \end{align}                  
	\end{defi}
	
	With this notion in hand, we prove corresponding lower bound scaling estimates which hold for a large class of homogeneous, constant coefficient, linear differential operators $\A(D)$.
	
	\begin{thm}\label{thm:lower_bound_p}
		Let $d, m \in \N, d \geq 2$. Let $\Omega \subset \R^d$ be an open and bounded Lipschitz domain.  Let $\A(D)$ be a homogeneous, constant coefficient, linear differential operator and $A,B \in \R^n$ such that
		$A-B \in \Lambda_{\A} \setminus I_{\A}$. Let $p(\xi)=|\AA(\xi)(A-B)|^2$ have the maximal vanishing order equal to $L \leq m$, cf. \cref{defi:index} and further assume $V =  p^{-1}(0)$ to be a finite union of vector spaces. For $\lambda \in (0,1)$ consider $F_\lambda = \lambda A + (1-\lambda)B$ and let $E^{\A}_{\epsilon}$ be as in
		\cref{eq:energy_total_gen} above. Then, there exist constants $C=C(\A(D), A, B, d, m, \Omega, V)>0$ and $\epsilon_0 = \epsilon_0(\A(D),\lambda,A,B,d,m,\Omega,V) > 0$ such that for any $\epsilon \in (0,\epsilon_0)$
		\begin{align*}
			\inf\limits_{\chi \in BV(\Omega;\{A,B\})}  E_{\epsilon}^{\A}(\chi;F_\lambda) \geq C \min\{1-\lambda,\lambda\}^2 \epsilon^{\frac{2L}{2L+1}}.
		\end{align*}
	\end{thm}

	Let us comment on this: Firstly, we highlight that for the special case of the higher order curl and its potential, the generalized symmetrized gradient, the lower bound from \cref{thm:lower_bound_p} coincides with the lower bounds from \cref{thm:scaling_2D_new}. Given the
	matching upper bounds for $d=2$, $\lambda = \frac{1}{2}$ and arbitrary tensor order $m\in \N$ these are indeed optimal. The estimates from \cref{thm:lower_bound_p} are obtained by a combination of general high frequency bounds, quantitative coercivity estimates away from the zero
	set of $p$ and low frequency bounds for which we use careful localization arguments. 
Apart from the setting of the higher order curl and $d=2$, $\lambda = \frac{1}{2}$, we do not know whether
	the bounds from \cref{thm:lower_bound_p} are optimal. We plan to study associated upper bound constructions in future work.
	Secondly, we point out that the assumption on $V$ is such that the theorem is applicable for the model operators $\curl$, $\curl \curl$, $\di$ (and their higher order generalizations). From a technical point of view, it allows for an easier splitting argument in Fourier space. Without the assumption that $V$ is a finite union of linear spaces, more complicated structures may arise. It is feasible that covering arguments can be used to reduce these to similar settings as for the vector space case. For the clarity of presentation, we however do not discuss this in the present article.

	\subsection{Relation to the literature}
	The two-well problem and more generally the $N$-well problem are intensively studied questions in the non-convex calculus of variations \cite{M1, DaM12, K1,KMS03, CK00,Ri18}. They are closely
	related to questions on pattern formation in various materials, including, for instance, shape-memory alloys. By now there is a large literature also on quantitative results for associated
	singular perturbation problems which build on the seminal works of \cite{KM92,KM94}, including, for instance, the articles \cite{AKKR22,CC15,CDMZ20,CZ16,CT05, CKZ17, CDPRZZ20, CO12, CO09,
		GZ23, KK11, KKO13,R16, RTZ19, RT22,RTrib22, RT23}.  While generalizations of the differential inclusions from materials science to more general linear differential operators were already
	studied in the context of compensated compactness \cite{T79,T93,T05,Mur81,FM99,Dac82}, renewed interest in the associated problems in the calculus of variations has recently arisen in the context of investigating
	structure conditions, e.g., in the context of Korn-type inequalities \cite{GLN22}, compensated compactness \cite{GR22, GRS22, KR22, Rai19}, regularity results \cite{CG22} and also in the context of $N$-well
	problems \cite{DPPR18,RRT23,ST23}.

	\subsection{Outline of the article}
	The remainder of the article is structured as follows: After briefly recalling and fixing notation in \cref{sec:notation}, in \cref{sec:lower} we turn to the derivation of lower
	scaling bounds. To this end, we first deduce lower bounds for the symbol and then translate these into scaling behaviour, identifying the maximal vanishing order of the associated symbol on the unit sphere as the determining ingredient
	for these estimates. In \cref{sec:upper} we complement these bounds with upper scaling bounds in the case of the generalized symmetrized gradient in two dimensions and specific boundary data but with general
	tensor order.

	\section{Notation and Preliminaries}
	\label{sec:notation}
	
	In this section we collect some background on the tensors under consideration. In particular, we recall a characterization for the higher order curl by the Saint-Venant
	conditions and compute the wave cone for the higher order curl and divergence.
	
	\subsection{Tensor notation}
	\label{sec:tensors}
	
	We denote the space of rank $m$-tensors by $T^m(\R^d) = (\R^d)^{\otimes m}$, the elements are thus multilinear maps $M:\prod_{j=1}^m \R^d \to \R$ 
	with components given by $M_{i_1 i_2 \dots i_m} = M(e_{i_1},e_{i_2},\dots,e_{i_m})$.  The subspace of symmetric $m$-tensors is denoted by $\S{m} \subset T^m(\R^d)$, i.e., they satisfy
	$M(v_1,\dots,v_m) = M(v_{\tau(1)},\dots,v_{\tau(m)})$ for any permutation $\tau \in \mathfrak{S}_m$.
	
	By $\sigma_{i_1 \dots i_m}$ we denote the \emph{symmetrization operator} in the indices $i_1,\dots,i_m$ which is defined as
	\begin{align*}
		\sigma_{i_1 \dots i_m}(M_{i_1 \dots i_m}) = \frac{1}{m!} \sum_{\tau \in \mathfrak{S}_m} M_{i_{\tau(1)} \dots i_{\tau(m)}}, \ M \in T^m(\R^d).
	\end{align*}
	Similarly, we define the \emph{alternation operators} $\alpha_{i_k i_l}$ as
	\begin{align}
	\label{eq:antisym}
		\alpha_{i_k i_l} (M_{i_1 \dots i_m}) = \frac{1}{2} \Big(M_{i_1 \dots i_m} - M_{i_1 \dots i_{k-1} i_l i_{k+1} \dots i_{l-1} i_k i_{l+1} \dots i_m}\Big),
	\end{align}
	where, without loss of generality, we have assumed that $i_k < i_l$.
	
	We say a mapping $u: \R^d \to \S{m}$ is a symmetrized derivative, if there is $v: \R^d \to \S{m-1}$ such that
	\begin{align}\label{eq:potential}
		u_{i_1\dots i_m} = [D^{sym} v]_{i_1 \dots i_m}:= \sigma_{i_1 \dots i_m}( \p_{i_1} v_{i_2 \dots i_m}) = \frac{1}{m!} \sum_{\tau \in \mathfrak{S}_m} \p_{i_{\tau(1)}} v_{i_{\tau(2)} \dots i_{\tau(m)}}.
	\end{align}
	We refer to $u$ as the \emph{symmetrized derivative} of $v$.

	Using the tensor product of vectors given by $[v^1 \otimes \dots \otimes v^m]_{i_1 \dots i_m} = \prod_{k=1}^m v^k_{i_k}$ for $v^1,\dots,v^m \in \R^d$, $i_1, \dots, i_m \in \{1,\dots,d\}$, we introduce the \emph{symmetric tensor product} of vectors by setting
	\begin{align}
		\label{eq:sym_prod}
		v^1 \odot \dots \odot v^m = \sigma_{1 \dots m} ( v^1 \otimes \dots \otimes v^m)
	\end{align}
	for $v^1,\dots, v^m \in \R^d$.
        Both definitions can be adapted for tensor products of tensors instead of vectors, as these elementary tensors form a spanning set, thus for $M,N \in T^m(\R^d)$, we can write
          \begin{align*}
            M = \sum_{i_1,\dots,i_m=1}^d M_{i_1 \dots i_m} e_{i_1} \otimes \dots \otimes e_{i_m}, \ N = \sum_{j_1, \dots, j_m = 1}^d N_{j_1 \dots j_m} e_{j_1} \otimes \dots \otimes e_{j_m}
          \end{align*}
          and therefore we have
          \begin{align*}
            M \otimes N = \sum_{i_1,\dots,i_m,j_1,\dots,j_m=1}^d M_{i_1 \dots i_m} N_{j_1 \dots j_m} e_{i_1} \otimes \dots \otimes e_{i_m} \otimes e_{j_1} \otimes \dots \otimes e_{j_m}.
          \end{align*}
          An analogous operation is defined for symmetric tensors (for which an additional symmetrization is necessary).
          Moreover this can also be defined for tensors of different order.
        For convenience of notation, for $e_k \in \R^d$ and $j \in \N$ we also set
        \begin{align*}
        e_k^{\odot j}:= e_k \odot \dots \odot  e_k,
        \end{align*}
        where the symmetrized product on the right hand side is $j$ times with itself.

          In order to simplify the notation, we will use standard notation for multiinidices $l \in \N^d$.
          For given $l =(l_1,\dots,l_d)\in \N^d$ the absolute value is given by $|l| = \sum_{j=1}^d l_j$, the factorial by $l! = \prod_{j=1}^d l_j!$ and the multinomial coefficient by $\binom{|l|}{l} = \frac{|l|!}{l!}$.
          Moreover, for a given vector $\xi \in \R^d$ we use the convention that $\xi^l: = \prod_{j=1}^d \xi_j^{l_j}$ and $\p^l = \p_1^{l_1} \dots \p_d^{l_d}$.
          
	\subsection{Example operators}
	\label{sec:Saint_venant}

	Using the alternation operators $\alpha_{i j}$ from \cref{eq:antisym}, we consider the generalized Saint-Venant compatibility operator as the first example of an $m$-th order operator $\A(D): C^\infty(\R^d;\S{m}) \to C^\infty(\R^d;T^{2m}(\R^d))$ defined by
	\begin{align} \label{eq:Operator} [\A(D) u]_{i_1j_1 \dots i_m j_m} = \alpha_{i_1 j_1} \circ \dots \circ \alpha_{i_m j_m} (\p^m_{j_1 \dots j_m} u_{i_1 \dots i_m}).
	\end{align}

We provide the explicit formulas for this operator in the case $m\in \{1,2,3\}$ and $d=2$.	
	
	\begin{example}[$d=2$]\label{ex:d-and-m}
		Fixing the spatial dimension $d=2$, the compatibility conditions, which are given by a system of equations for general dimension $d$, become particularly transparent. More precisely, due to
		symmetry (see \cref{eq:sym} below), they simplify to the single equation $[\A(D)u]_{12\dots12} = 0$. Considering tensors of order $m=1,2,3$, we obtain the following compatibility conditions.
		\begin{itemize}
			\item \underline{$m=1$}: In this case we compute that $[\A(D)u]_{ij} =\frac{1}{2}( \p_j u_i - \p_i u_j)$ and thus
			\begin{align*}
				\A(D)u = 0 \text{ if and only if } \p_1 u_2 - \p_2 u_1 = 0.
			\end{align*}
			This is the well-known case of the characterization of gradients by means of Poincar\'e's lemma.
			\item \underline{$m=2$}: In the case of second order tensors we observe that
			\begin{align*} [\A(D)u]_{ikjl} = \frac{1}{4} \big( \p^2_{kl} u_{ij} + \p^2_{ij} u_{kl} - \p^2_{il} u_{kj} - \p^2_{kj} u_{il} \big).
			\end{align*}
			Hence,
			\begin{align*}
				\A(D)u = 0 \text{ if and only if } \p^2_{11} u_{22} + \p^2_{22} u_{11} - 2 \p^2_{12}u_{12} = \curl \curl u = 0.
			\end{align*}
			This corresponds to the classical characterization of the symmetrized gradient by means of the Saint-Venant conditions.
			\item \underline{$m=3$}: For third order tensors, also only one independent equation remains:
			\begin{align*}
				&\A(D)u = 0 \text{ if and only if } \\
				&[\A(D)u]_{121212} = \frac{1}{8} \big( \p^3_{222} u_{111} - \p^3_{111} u_{222} + 3 \p^3_{211} u_{122} - 3 \p^3_{221} u_{112}\big) = 0.
			\end{align*}
			This is the characterization of being a symmetrized derivative. We emphasize that analogous characterizations can be obtained for tensors of arbitrary order and dimension.
		\end{itemize}
	\end{example}
	
	Also beyond the case $d=2$, the Saint-Venant operator characterizes symmetrized derivatives:
	A function $f \in C^\infty_c(\R^d;\S{m})$ fulfills $\A(D)f = 0$ if and only if $f$ is a symmetrized derivative \cite[Thm. 2.2.1, Eq. (2.4.6), (2.4.7)]{Sha94}.
In what follows, we will therefore also refer to the Saint-Venant operator as a higher order curl operator.
        
          By definition of $\A(D)$, we have the following (anti)symmetries:
          \begin{align}
          \label{eq:sym}
          \begin{split}
            [\A(D)u]_{i_{\tau(1)} j_{\tau(1)} i_{\tau(2)} j_{\tau(2)} \dots i_{\tau(m)} j_{\tau(m)}}& = [\A(D) u]_{i_1 j_1 \dots i_m j_m} \text{ for all } \tau \in \mathfrak{S}_m, \\
            [\A(D)u]_{i_1j_1\dots i_mj_m} & = - [\A(D)u]_{j_1 i_1 i_2 j_2 \dots i_m j_m}.
            \end{split}
          \end{align}

          In what follows, we will make use of these symmetries to further determine the symbol of $\A(D)$, as defined in \cref{eq:symbol}.
          For $\xi \in \R^d$, componentwise, it is given by
	\begin{align} \label{eq:Symbol_SVO}
		\begin{split}
			[\AA(\xi) M]_{i_1 j_1 \dots i_m j_m} & = \alpha_{i_1 j_1} \circ \dots \circ \alpha_{i_m j_m} (\xi_{j_1} \dots \xi_{j_m} M_{i_1 \dots i_m}) \\
			& = \alpha_{i_1 j_1} \circ \dots \circ \alpha_{i_m j_m} (M(\xi_{j_1} e_{i_1},\dots, \xi_{j_m} e_{i_m})) \\
			& = 2^{-m}  M(\xi_{j_1} e_{i_1} - \xi_{i_1} e_{j_1}, \dots, \xi_{j_m} e_{i_m} - \xi_{i_m} e_{j_m})
		\end{split}
	\end{align}
	for $M \in \S{m}$.  Moreover, by multilinearity for any orthonormal basis $v_1,\dots,v_d$ it holds
	\begin{align} \label{eq:InvarianceBasis}
          \begin{split}
            [\AA(\xi)M]&(v_{i_1},v_{j_1},\dots,v_{i_m},v_{j_m}) \\
          & = 2^{-m} M\Big((\xi \cdot v_{j_1}) v_{i_1} - (\xi \cdot v_{i_1}) v_{j_1},\dots, (\xi \cdot v_{j_m}) v_{i_m} - (\xi \cdot
		v_{i_m}) v_{j_m}\Big).
          \end{split}
	\end{align}
	Using these observations, we rewrite the symbol of the higher order curl in a concise way.
	
		\begin{lem}\label{lem:SVsymb}
			Let $M = a_1 \odot \dots \odot a_m$ for $a_1,\dots,a_m \in \R^d$.  For $\xi \in \R^d$ and $\AA$ as in \cref{eq:Symbol_SVO}, we have
			\begin{align} \label{eq:AlternativeSymbol}
                          \AA(\xi)M = (a_1 \circleddash \xi) \odot \dots \odot (a_m \circleddash \xi).
			\end{align}
			Here we use the symbol $\circleddash$ to denote the antisymmetric tensor product of two vectors, i.e.~$a \circleddash b = \frac{1}{2}( a \otimes b - b \otimes a)$ for $a,b \in \R^d$, and consider the symmetrized tensor product of tensors:
                          \begin{align*}
                            (a_1 \circleddash \xi) & \odot \dots \odot (a_m \circleddash \xi) := \frac{1}{m!} \sum_{\tau \in \mathfrak{S}_m} (a_{\tau(1)} \circleddash \xi) \otimes \dots \otimes (a_{\tau(m)} \circleddash \xi) \\
                            & = \frac{1}{m!} 2^{-m} \sum_{\tau \in \mathfrak{S}_m} (a_{\tau(1)} \otimes \xi - \xi \otimes a_{\tau(1)}) \otimes \dots \otimes (a_{\tau(m)} \otimes \xi - \xi \otimes a_{\tau(m)}).
                        \end{align*}

		\end{lem}
		
		\begin{proof}
			By \cref{eq:Symbol_SVO}, we have
			\begin{align*}
				[\AA(\xi)M]_{i_1j_1\dots i_m j_m} & = \frac{1}{m!} \sum_{\tau \in \mathfrak{S}_m} \Big( \prod_{k=1}^m \frac{1}{2} a_{\tau(k)}\cdot (\xi_{j_k} e_{i_k} - \xi_{i_k} e_{j_k}) \Big) \\
				&= \frac{1}{m!} \sum_{\tau \in \mathfrak{S}_m} \Big( \prod_{k=1}^m \frac{1}{2} ([a_{\tau(k)} \otimes \xi]_{i_k j_k} - [\xi \otimes a_{\tau(k)}]_{i_k j_k} ) \Big) \\
				& = \frac{1}{m!} \sum_{\tau \in \mathfrak{S}_m} \Big( \prod_{k=1}^m [a_{\tau(k)} \circleddash \xi]_{i_k j_k} \Big) \\
				&= \frac{1}{m!} \sum_{\tau \in
					\mathfrak{S}_m} \Big[ (a_{\tau(1)} \circleddash \xi) \otimes \dots \otimes (a_{\tau(m)} \circleddash \xi) \Big]_{i_1 j_1 \dots i_m j_m}.
			\end{align*}
			This shows the claim.
		\end{proof}
		          
            Motivated by \cite[Appendix B]{RRT23}, i.e., by the fact that lower bounds on $E_{\epsilon}^{\A}$ for any $m$-th order homogeneous, constant coefficient, linear differential operator $\A(D)$ can be deduced by lower bounds for the $m$-th order divergence, as a second model example of an $m$-th order operator we consider the $m$-th order divergence
            \begin{gather} \label{eq:DivergenceOp}
              \begin{aligned}
                \mathcal{B}(D) = \di^m &: C^\infty(\R^d;\R^k \otimes \S{m}) \to C^\infty(\R^d;\R^k), \\
                [\di^m u]_j & := \sum_{1\leq i_1 \leq \dots \leq i_m \leq d} \partial_{i_1 \dots i_m}^m u_{j i_1 \dots i_m}, \ j \in \{1,\dots,k\},
              \end{aligned}
            \end{gather}
            for some integer $k\ge1$.
            It is straightforward to extend this definition (and all the results obtained in the sequel for $\di^m$) to $m$-th order divergence-type operators acting on fields $u:\R^d\to W\otimes\S{m}$, where $W$ is a (real) $k$-dimensional vector space. We denote the symbol of $\mathcal{B}(D)$ by $\mathbb{B}$.

            Similarly as in \cref{lem:SVsymb}, we give a precise formulation for the symbol of $\mathcal{B}(D)$ on a basis of $\R^k \otimes \S{m}$.
            \begin{lem} \label{lem:SymbolDiv}
              Let $v\in\R^k$ and let $M = v \otimes e_1^{\odot l_1} \odot \dots \odot e_d^{\odot l_d}$ for a partition $\sum_{j=1}^d l_j = m$. The symbol of $\mathcal{B}(D)$ is given by
              \begin{align*}
                \mathbb{B}(\xi)M = \frac{l_1! l_2! \dots l_d!}{m!} \prod_{j=1}^d \xi_j^{l_j} v = \binom{m}{l}^{-1} \xi^l v,
              \end{align*}
              with $l = (l_1,l_2,\dots,l_d) \in \mathbb{N}^d$, $\xi \in \R^d$.
            \end{lem}
            \begin{proof}
              Let $M = v \otimes e_1^{\odot l_1} \odot \dots \odot e_d^{\odot l_d}$ for some $v \in \R^k$. Using that for $1 \leq i_1 \leq i_2 \leq \dots \leq i_m \leq d$
              \begin{align*}
                [M]_{j i_1 \dots i_m} = \begin{cases}
                                          v_j \frac{l_1! l_2! \dots l_{d}!}{m!} & \begin{aligned} &i_1,\dots,i_{l_1} = 1, i_{l_1+1}, \dots, i_{l_1+l_2} = 2, \dots, \\ &i_{m-l_d+1}, \dots, i_{m} = d, \end{aligned} \vspace{5pt}\\
                                          0 & \text{else,}
                                        \end{cases}
              \end{align*}
              we obtain
                        \begin{align*}
                          [\mathbb{B}(\xi)M]_j & = \sum_{1 \leq i_1 \leq i_2 \leq \dots \leq i_m \leq d} \xi_{i_1} \xi_{i_2} \dots \xi_{i_m} M_{ji_1 \dots i_m} = \xi_1^{l_1} \xi_{2}^{l_2} \dots \xi_{d}^{l_d} v_j \frac{l_1! \dots l_d!}{m!}\\
                          & = v_j \binom{m}{l}^{-1} \xi^l.
                          \end{align*}
            \end{proof}

          In what follows, we will consider the extensions of $\A(D), \mathcal{B}(D)$ to distributional derivatives by duality.

	\subsection{Computation of the wave cones for symmetrized derivatives and the higher order divergence}
	\label{sec:symm_deriv}
	With the above discussion in mind, in what follows, we consider the differential operators given by \cref{eq:Operator,eq:DivergenceOp}.
		 We next identify the associated wave cones.
	
	\begin{lem}[Higher order curl] 
	\label{lem:WaveCone} 
	For $\xi \in \R^d$ let $\AA(\xi)$ be given by \cref{eq:Symbol_SVO}, where the operator $\A(D)$ is given in \cref{eq:Operator}.  The kernel of $\AA(\xi)$ is given by
		\begin{align*}
			\ker \AA(\xi) & = \vspan \Big\{ \sigma_{1 \dots m} (v_{1} \otimes \dots \otimes v_{m}): v_{i} \in \vspan(\xi) \text{ for some } i \in \{1,\dots,m\}\Big\} \\
			& = \vspan\Big\{ \xi \odot a_2 \odot \dots \odot a_m: a_2,\dots,a_m \in \R^d \Big\}.
		\end{align*}
	\end{lem}
	
	\begin{proof}
		We prove the claim of the lemma by showing that being in the span of $\{\xi \odot a_2 \odot \dots \odot a_m: a_2,\dots,a_m \in \R^d \}$ is both necessary and sufficient for being an element
		of the kernel.
		
		For simplicity we may assume that $\xi = e_1$ as by the homogeneity of $\AA$ we have $\ker \AA(\xi) = \ker \AA(\frac{\xi}{|\xi|})$ and further we can choose an orthonormal basis
		$v_1,\dots,v_d$ such that $v_1 = \frac{\xi}{|\xi|}$, and thus
		\begin{align*} [\AA(\xi)M](v_{i_1},v_{j_1},\dots,v_{i_m},v_{j_m}) = |\xi|^m[\AA(v_1)M](v_{i_1},v_{j_1},\dots,v_{i_m},v_{j_m}).
		\end{align*}
		This change of basis can be seen in \cref{eq:InvarianceBasis}.
		
		Having fixed this, we seek to show that
		\begin{align*}
			\ker \AA(e_1) = \vspan\Big\{ e_1 \odot a_2 \odot \dots \odot a_m: a_2,\dots,a_m \in \R^d\Big\}.
		\end{align*}
		
		First, let $M = e_1 \odot a_2 \odot \dots \odot a_m$. The inclusion $\vspan\Big\{ e_1 \odot a_2 \odot \dots \odot a_m: a_2,\dots,a_m \in \R^d\Big\} \subset \ker \AA(e_1) $ is
		then immediate: Indeed, by \cref{eq:AlternativeSymbol}, we see that $\AA(\xi)M = 0$, since $e_1 \circleddash e_1 = 0$. 
		
		For the converse inclusion, we assume that $M \in \ker \AA(e_1)$. Using \cref{eq:Symbol_SVO}, we consider the components given by $i_1,\dots,i_m \neq 1$:
			\begin{align*} [\AA(e_1)M]_{i_1 1 \dots i_m 1} = M(e_{i_1}, \dots, e_{i_m}) = M_{i_1 \dots i_m} = 0.
			\end{align*}
			Furthermore, as $e_{k_1} \odot \dots \odot e_{k_m}$ for $1 \leq k_1 \leq k_2 \leq \dots \leq k_m \leq d$ forms a basis of $\S{m}$ and as we have seen, the only non-vanishing components of $M$ are those with at least one $1$ in the index, we can write 
			\begin{align*}
				M & = \sum_{i_1,i_2,\dots,i_m = 1}^d M_{i_1 i_2 \dots i_m} e_{i_1} \otimes e_{i_2} \otimes \dots \otimes e_{i_m} \\
				& = \sum_{1 \leq i_1 \leq i_2 \leq \dots \leq i_m \leq d} \binom{m}{\sum_{p=1}^m e_{i_p}} M_{i_1 \dots i_m} e_{i_1} \odot \dots \odot e_{i_m}\\
				& = \sum_{1 \leq i_2 \leq \dots \leq i_m \leq d} \binom{m}{e_1 + \sum_{p=2}^m e_{i_p}} M_{1 i_2 \dots i_m} e_1 \odot e_{i_2} \odot \dots \odot e_{i_m}.
			\end{align*}
			This shows that indeed $M \in \vspan\{e_1 \odot a_2 \odot \dots \odot a_m: a_2,\dots,a_m \in \R^d\}$.
		\end{proof}
		
		In concluding this section, we also consider the $m$-th order divergence and compute the structure of its wave
		cone.
		
		\begin{lem}[Higher order divergence] \label{rmk:Divergence} Let $\mathcal{B}(D)$ be the $m$-th order divergence given by \cref{eq:DivergenceOp}.			
                  Then the wave cone of $\mathcal{B}(D)$ is given by ($\xi \in \R^d$)
			\begin{align*}
				\ker \mathbb{B}(\xi) = \vspan \Big\{ v \otimes (a_1 \odot \dots \odot a_m): v \in \R^k, \ a_j \cdot \xi = 0 \text{ for some } j \in \{1,\dots,m\} \Big\}.
			\end{align*}
                      \end{lem}
                      
                      \begin{proof}
                        We show the claim by using that $\ker \mathbb{B}(\xi) = (\operatorname{ran} \mathbb{B}(\xi)^\ast)^\perp$, as the adjoint has a simple structure.
                        
                  Indeed, considering the scalar product on symmetric
				tensors $S,T \in \S{m}$ given by $S \cdot T = \sum_{1\leq i_1 \leq \dots \leq i_m \leq d} S_{i_1 \dots i_m} T_{i_1 \dots i_m}$,
                  the adjoint is given by 
			\begin{align*}
				\mathbb{B}(\xi)^\ast w = w \otimes \xi \otimes \xi \otimes \dots \otimes \xi = w \otimes \xi^{\odot m} \in \R^k \otimes \S{m}, \  w \in \R^k.
			\end{align*}
			Thus the kernel is given by
			\begin{align*}
                          \ker \mathbb{B}(\xi) & = \{w \otimes \xi^{\odot m}: w \in \R^k\}^\perp  \\
                          & = \vspan\Big\{v \otimes (a_1 \odot \dots \odot a_m): v \in \R^k, a_j \cdot \xi = 0 \mbox{ for some } j\Big\}.
			\end{align*}
                        This shows the statement.
		\end{proof}

		\section{Lower Bound Scaling Results}
		\label{sec:lower}
		
		With the characterization of the wave cones for the higher order curl and the higher order divergence in hand, in this section we turn to the proof of (general) lower scaling
		bounds. The core of this consists of an adaptation of the lower bound argument from \cite{RRT23} allowing to deal with rather general zero sets consisting of a union of linear subspaces, see \cref{sec:scaling}.
                
                  To be more precise, we have the following lower bound of the energy for a general homogeneous, constant coefficient, linear operator $\A(D)$ and two wells $A,B \in \R^n$ with $F_\lambda = \lambda A + (1-\lambda) B$:
                Writing $\chi = f(A-B) + F_\lambda \in BV(\Omega;\{A,B\})$ with $f \in BV(\Omega;\{1-\lambda,-\lambda\})$ (extended to $\R^d$ by zero), \cite[Corollary 3.2]{RRT23} states
                \begin{align} \label{eq:LowerBoundRRT}
                  E^{\A}_{\epsilon}(\chi;F_\lambda) \geq C \Big(\int_{\R^d} \Big|\AA(\frac{\xi}{|\xi|})(A-B)\Big|^2 |\hat{f}|^2 d\xi + \epsilon \int_\Omega |\nabla f| \Big).
                \end{align}
                Here and in the following, for every $f\in L^2(\R^d)$ we denote its Fourier transform by
                \begin{align*}
                \hat f(\xi) := (2\pi)^{-\frac{d}{2}}\int_{\R^d} e^{-i \xi \cdot x} f(x)dx.
                \end{align*}
                Moreover, we recall the definition of the distance function (cf. \cref{eq:dist} in \cref{defi:index}), and note that $\dist_V$ is positively $1$-homogeneous for $V$ being a finite union of vector spaces.
                Now for the polynomial $p(\xi) = |\AA(\xi)(A-B)|^2$ having the maximal vanishing order $L$ (cf. \cref{defi:index}), we can further bound
                \begin{align*}
                  E^{\A}_\epsilon(\chi;F_\lambda) \geq C \Big( \int_{\R^d} \frac{\dist_V(\xi)^{2L}}{|\xi|^{2L}} |\hat{f}|^2 d\xi + \epsilon \int_\Omega |\nabla f| \Big).
                \end{align*}
                Thus, to control the lack of coercivity near the zero set $V$ of $p$, we can consider the multiplier to be given by $\dist_V(\frac{\xi}{|\xi|})^{2L}$.

		\subsection{Scaling results}
		\label{sec:scaling}
		
                In this subsection, we provide the central estimates for our lower scaling bounds in the case that the symbol of the operator vanishes on a union of vector spaces. Let $\Omega\subset\R^d$ be a bounded set. We will work with functions $f\in L^2(\Omega)$ which we
		identify with their extensions by zero to the full space without mention.

		As a central result in this section, we prove the following bounds.
		\begin{prop}\label{prop:key_estimate}
			Let $d \in \N$, $d \geq 2$. Let $L$ be a positive integer and let $\Omega \subset \R^d$ be an open, bounded Lipschitz domain. Suppose that $V\subset\R^d, V \neq 0$ is a union of finitely many linear spaces of dimension at most $d-1$. Then the following estimates hold for every $\eta > 1$:
			\renewcommand{\labelenumi}{\textbf{\theenumi}} \renewcommand{\theenumi}{(\roman{enumi})}
			\begin{enumerate}
				\item\label{itm:a} For any $\delta \in (0,1)$ there exists $\alpha  = \alpha(\delta,\Omega,V) \in (0,1)$ such that
				\begin{align*}
				\int_{\dist_V(\xi)\leq \alpha} |\hat f|^2 d\xi \leq \delta \int_{\R^d} |\hat f|^2d\xi \quad\text{for }f\in L^2(\Omega).
                                \end{align*}
                              \item\label{itm:b} For all $\alpha>0$
				\begin{align*}
				\int_{\dist_V(\xi)\geq \alpha,|\xi|\leq \eta}|\hat f|^2 d\xi \leq \left(\frac{\eta}{\alpha}\right)^{2L}\int_{\R^d}\frac{\dist_V(\xi)^{2L}}{{|\xi|}^{2L}}|\hat f|^2d\xi\quad\text{for
				}f\in L^2(\Omega).
                                \end{align*}
                              \item\label{itm:c} We have that
                                \begin{align*}
				\hspace{-19pt} \int_{|\xi|\geq\eta}|\hat f|^2 d\xi\leq C(d) \Vert f \Vert_\infty \eta^{-1} \big(\int_\Omega|\nabla f|+\Vert f \Vert_{\infty} \Per(\Omega)\big) \text{ for }f\in L^{\infty}(\Omega)\cap BV(\Omega).
                                \end{align*}
                              \end{enumerate}
			Summing the three estimates, with the constant $\alpha$ from \cref{itm:a}, and absorbing the right hand side of \cref{itm:a}, we obtain that for any  $f\in L^{\infty}(\Omega)\cap BV(\Omega)$:
			\begin{align*}
			\int_{\R^d}|\hat f|^2 d\xi\leq C \left(\frac{\eta}{\alpha}\right)^{2L}\int_{\R^d}\frac{\dist_V(\xi)^{2L}}{{|\xi|}^{2L}}|\hat f|^2d\xi+C\eta^{-1} \Vert f \Vert_\infty\big(\int_\Omega|\nabla
			f|+\Vert f \Vert_\infty \Per(\Omega)\big),
			\end{align*}
			where the constant $C>0$ depends on $\delta,d$.
		\end{prop}

		The proof of \cref{itm:b} is immediate from the definition of the domain of integration on the left hand side. The proof of \cref{itm:c} is known from \cite{KKO13}. The proof of
		\cref{itm:a} requires more attention, so we extract a relevant slicing lemma.
		\begin{lem}\label{lem:slicing}
			Let $f\in L^2(\Omega)$, $1\leq s\leq d$ be a positive integer. Write $\xi=(\xi',\xi'')$ for $\xi'\in\R^s, \xi'' \in \R^{d-s}$. We have that for $\mathcal H^{d-s}$-a.e. $\xi''\in\R^{d-s}$
			\begin{align} \label{eq:slicing}
                          {\mathrm{ess\,sup}_{\xi'\in\R^s}}|\hat f(\xi',\xi'')|^2\leq \Big(\frac{\diam\,\Omega}{2\pi}\Big)^{s}\int_{\R^s}|\hat f(\xi',\xi'')|^2 d\xi'.
			\end{align}
		\end{lem}
		\begin{proof}
			It is instructive to first cover the case $s=d$. Then we have that
			\begin{align}\label{eq:Linfty-L2}
                          \begin{split}
				\|\hat f\|_{L^\infty} &\leq (2\pi)^{-\frac{d}{2}}\|f\|_{L^1} \leq (2\pi)^{-\frac{d}{2}} \mathcal L^d( \Omega)^{1/2}\|f\|_{L^2} \\
                                                      & \leq (2\pi)^{-\frac{d}{2}} (\diam \,\Omega)^{d/2}\|f\|_{L^2} = \left(\frac{\diam \,\Omega}{2 \pi}\right)^{d/2}\|\hat f\|_{L^2}.
                                                        \end{split}
			\end{align}
			
			To show the claim for $s \leq d$, we first notice that by Fubini's theorem, the right hand side in \cref{eq:slicing} is finite for $\mathcal H^{d-s}$ a.e. $\xi''$. The required estimate follows from
			\cref{eq:Linfty-L2} applied to $g=\mathcal F_{\xi''}f(\,\cdot\,,\xi'')$, provided we show that $g$ has compact support in $\R^s$ of diameter at most $\diam\, \Omega$. Let $\omega$ be the
			projection of $\Omega$ on $\R^s$, so that $\diam\,\omega\leq \diam\,\Omega$. Moreover, $f(x',x'')=0$ for all $x'\in\R^s\setminus\omega$. Therefore,
			\begin{align*}
			g(x')=(2\pi)^{-\frac{d-s}{2}}\int_{\R^{d-s}}f(x',x'')e^{-ix''\cdot\xi''}dx''=0\quad \text{for } x'\in\R^s\setminus\omega
                        \end{align*}
			so $g$ is supported inside $\omega$. The proof is complete.
		\end{proof}
		We can now return to the proof of \cref{prop:key_estimate} above.
		\begin{proof}[Proof of \cref{prop:key_estimate}]
			It remains to prove \cref{itm:a}. To do this, we first assume that $V$ is a linear space, which we identify with $\R^{d-s}$ for some $1\leq s\leq d-1$. Writing coordinates
			$\R^d\ni\xi=(\xi',\xi'')$ with $\xi'\in\R^s = V^\perp, \xi'' \in \R^{d-s}$, we control, with $\alpha>0$ to be determined,
			\begin{align*}
                          \int_{|\xi'|\leq \alpha}|\hat f|^2d\xi&=\int_{\R^{d-s}}\int_{|\xi'|\leq \alpha}|\hat f(\xi',\xi'')|^2 d\xi'd\xi''\\
                          & \leq C(s)\alpha^s\int_{\R^{d-s}}\mathrm{ess\,sup}_{\xi' \in\R^s}|\hat f(\xi',\xi'')|^2 d\xi''\\
				&\leq C(s)(\frac{1}{2\pi} \alpha\diam\,\Omega)^s\int_{\R^{d-s}}\int_{\R^s}|\hat f(\xi',\xi'')|^2d\xi' d\xi'',
			\end{align*}
			where to obtain the last inequality we use \cref{lem:slicing}; here $C(s)$ denotes the area of the $s-1$ dimensional unit sphere.  This is enough to conclude the proof of this case by
			taking $\alpha = \alpha(\delta,\Omega,V)$ small enough.
			
			In the general case of a finite union of linear spaces, i.e., $V=\bigcup_{j=1}^NV_j$, where each $V_j$ is a linear space, by the previous step, there exists $\alpha = \alpha(\delta,\Omega,V)>0$ such that
			\begin{align*}
			\int_{\dist_{V_j}(\xi)\leq \alpha}|\hat f|^2d\xi\leq \frac{\delta}{N}\int_{\R^d}|\hat f|^2d\xi.
                        \end{align*}
			It follows that
			\begin{align*}
			\int_{\dist_V(\xi)\leq \alpha}|\hat f|^2d\xi\leq \sum_{j=1}^N \int_{\dist_{V_j}(\xi)\leq \alpha}|\hat f|^2d\xi\leq \sum_{j=1}^N \frac{\delta}{N}\int_{\R^d}|\hat f|^2d\xi,
                        \end{align*}
			which concludes the proof.
		\end{proof}
	
With the results of \cref{prop:key_estimate} in hand, we turn to the first lower scaling bounds.		
		
			\begin{prop}[Lower scaling bounds] 
			\label{thm:LowerScalingUnionVS} 
			Let $d,L \in \N, d \geq 2$. Let $\Omega \subset \R^d$ be an open, bounded Lipschitz domain.  Let $V \subset \R^d, V \neq 0$ be a union of
				finitely many linear spaces of dimension at most $d-1$.  For $f \in BV(\R^d;\{-\lambda,0,1-\lambda\})$ for $\lambda \in (0,1)$ with $f = 0$ in $\R^d \setminus \overline{ \Omega}$, $f \in \{1-\lambda,-\lambda\}$ in $\Omega$, we consider the energies given by
				\begin{align*}
					\tilde{E}_{el}(f) := \int_{\R^d} \frac{\dist_V(\xi)^{2L}}{|\xi|^{2L}} |\hat{f}|^2 d\xi, \ \tilde{E}_{surf}(f) := \int_{\Omega} |\nabla f|.
				\end{align*}
				Then there exist $\epsilon_0 = \epsilon_0(d,\lambda,L,\Omega, V) > 0, C = C(d,L,\Omega, V) > 0$ such that for $\epsilon \in (0,\epsilon_0)$ we have the following lower bound
				\begin{align*}
					\tilde{E}_{\epsilon}(f) := \tilde{E}_{el}(f) + \epsilon \tilde{E}_{surf}(f) \geq C \min\{1-\lambda,\lambda\}^2 \epsilon^{\frac{2L}{2L+1}}.
				\end{align*}
			\end{prop}
    			\begin{proof}
				As $V$ is a union of finitely many linear spaces of dimension at most $d-1$, we can apply \cref{prop:key_estimate} for any $\eta >1$. Thus, as $\Vert f \Vert_\infty \leq 1$, there exist a constant $C = C(d)>0$
				independent of $\eta$ and $\alpha = \alpha(\Omega,V) \in (0,1)$ such that
				\begin{align*}
					\int_{\R^d} |\hat{f}|^2 d\xi & \leq C \Big( \left(\frac{\eta}{\alpha}\right)^{2L} \tilde{E}_{el}(f) + \eta^{-1} \tilde{E}_{surf}(f) + \eta^{-1} \Per(\Omega) \Big) \\
					& \leq C \alpha^{-2L} \Big( \eta^{2L} \tilde{E}_{el}(f) + (\eta \epsilon)^{-1} \epsilon \tilde{E}_{surf}(f) + \eta^{-1} \Per(\Omega) \Big).
				\end{align*}
				Taking now $\eta = \epsilon^{-\frac{1}{2L+1}} >1$, it follows that
				\begin{align*}
					\int_{\R^d} |\hat{f}|^2 d\xi & \leq C \epsilon^{-\frac{2L}{2L+1}} \tilde{E}_\epsilon(f) + C \epsilon^{\frac{1}{2L+1}} \Per(\Omega),
				\end{align*}
				for some constant $C = C(d,L,\Omega,V)>0$.
				As $f \in L^2(\Omega;\{1-\lambda,\lambda\})$ we can bound the $L^2$ norm from below, thus by Plancherel's identity we infer
				\begin{align*}
					\int_{\R^d} |\hat{f}|^2 d\xi = \int_{\R^d} |f|^2 dx \geq \min\{1-\lambda,\lambda\}^2 |\Omega|.
				\end{align*}
				Choosing now $\epsilon_0 = \epsilon_0(d,\lambda,L, \Omega, V)$ such that $C \epsilon_0^{\frac{1}{2L+1}} \Per(\Omega) \leq \frac{1}{2} \min\{1-\lambda,\lambda\}^2 |\Omega|$, we obtain
				\begin{align*}
					C^{-1} \frac{1}{2} \min\{1-\lambda,\lambda\}^2 |\Omega|  \epsilon^{\frac{2L}{2L+1}} \leq \tilde{E}_\epsilon(f),
				\end{align*}
				which is the desired inequality.
                              \end{proof}

  \cref{thm:LowerScalingUnionVS} directly leads to the proof of \cref{thm:lower_bound_p}.

                              \begin{proof}[Proof of \cref{thm:lower_bound_p}]
                                By definition of the maximal vanishing order and from \cref{eq:LowerBoundRRT}, there exists a constant $C = C(\A(D),A,B)>0$ such that
                                \begin{align*}
                                  E_{\epsilon}(\chi;F_\lambda) \geq C \Big( \int_{\R^d} \frac{\dist_V(\xi)^{2L}}{|\xi|^{2L}} |\hat{f}|^2 d\xi + \epsilon \int_\Omega |\nabla f|\Big).
                                \end{align*}
                                Here $f \in BV(\Omega;\{1-\lambda,-\lambda\})$ is determined by $\chi = (A-B)f + F_\lambda$ and extended to $\R^d$ by zero.
                                For this we can apply \cref{thm:LowerScalingUnionVS} as by assumption $V$ is a finite union of vector spaces and $V \neq 0$ as $A-B \in \Lambda_\A$.
                                This shows the desired claim.
                              \end{proof}

		\subsection{Applications} 
		\label{sec:Applications} 
		As consequences of the estimates from the previous section, we can deduce lower scaling bounds for the higher order curl and the higher order divergence operators. Both fall into the class of operators for which \cref{prop:key_estimate} is applicable. In order to infer this, we begin by providing lower bound estimates for the associated multipliers.
                
		\begin{lem}[The higher order curl]
			\label{lem:LowerBoundMultipl}
			Let $d,m \in \N, d \geq 2$. Let $\A(D):C^\infty(\R^d; \allowbreak \S{m}) \to C^\infty(\R^d;T^{2m}(\R^d))$ be the operator from \cref{eq:Operator} with its symbol $\AA(\xi)$ given in \cref{eq:Symbol_SVO} and for $\sum_{j=1}^d l_j = m$ let $M = e_1^{\odot l_1} \odot \dots \odot e_d^{\odot l_d}$, $V := \{\xi \in \R^d: \AA(\xi)M = 0\} $, and $L = \max_{j=1,\dots,d} l_j  \leq m$.
			Then,
\begin{align}
\label{eq:zero_set}
V = \bigcup_{j: l_j \neq 0} \vspan(e_j),
\end{align}			
			and there exists a constant $C = C(d,m) > 0$ such that
			\begin{align}
			\label{eq:LowerSymbol1}
				|\AA(\frac{\xi}{|\xi|})M|^2 \geq C \frac{\dist_V(\xi)^{2L}}{|\xi|^{2L}}.
			\end{align}
		\end{lem}
		
		\begin{proof}

We note that the equality in \cref{eq:zero_set} follows from the characterization of the kernel of $\AA(\xi)$ in \cref{lem:WaveCone}. Indeed, by the assertion of \cref{lem:WaveCone} for $M = e_1^{\odot l_1} \odot \dots \odot e_d^{l_d}$ the roots of $\AA(\xi)M=0$ are given by $\xi \in \bigcup_{j: l_j \neq 0}\vspan(e_j)$.

In order to deduce \cref{eq:LowerSymbol1} it suffices to prove that
				\begin{align*}
					\inf_{\xi \in \mathbb{S}^{d-1} \setminus V} \frac{\Big| \AA(\xi) M \Big|^2}{\dist_V(\xi)^{2L}} > C \geq 0
				\end{align*}
				for some constant $C > 0$.  
				
				To show this claim, we fix $\delta \in (0,\frac{1}{2})$
				such that $\dist_V(\xi) \geq \delta$ for
				$\xi \in \mathbb{S}^{d-1}$ implies $|\AA(\xi)M|^2 \geq C$ for some constant $C = C(\delta, \AA) > 0$.  We consider two cases for $\xi \in \mathbb{S}^{d-1} \setminus V$.\\
				Firstly, if $\dist_V(\xi) \geq \delta$, then it holds
					\begin{align*}
						\frac{\Big| \AA(\xi) M \Big|^2}{\dist_V(\xi)^{2L}} \geq \Big| \AA(\xi)M \Big|^2 \geq C > 0.
					\end{align*}
					It thus remains to consider the case that $\dist_V(\xi) < \delta$. In this case there exists $j \in \{1,\dots,m\}$ such that $l_j \neq 0$ with $\dist_V(\xi)^2 = 1 - \xi_j^2$ and therefore also $\min\{|\xi \pm e_j|\} < \varepsilon(\delta)$ with $\varepsilon(\delta) \rightarrow 0$ as $\delta \rightarrow 0$.  Without loss of generality $|\xi - e_j|^2 < \varepsilon$.  Thus,
					there exists $w \perp e_j, |w|=1$ and $\rho \in (0,\varepsilon)$ such that $\xi = \frac{e_{j} + \rho w}{\sqrt{1+\rho^2}}$.  In the following we
					write $\tilde{\xi} = e_j + \rho w$.
					
					Using the structure of $M = e_1^{\odot l_1} \odot \dots \odot e_d^{\odot l_d}$, \cref{eq:AlternativeSymbol}, and that $e_j \circleddash \tilde{\xi} = \rho (e_j \circleddash w)$, we calculate
					\begin{align*}
						\AA(\tilde{\xi})M = \rho^{l_j} (e_1 \circleddash \tilde{\xi})^{\odot l_1} \odot \dots \odot (e_j \circleddash w)^{\odot l_j} \odot \dots \odot (e_d \circleddash \tilde{\xi})^{\odot l_d}.
					\end{align*}
					Moreover, as $\rho < 1$ and as $l_j \neq 0$, we know $\dist_V(\xi) = \rho$ and therefore
					\begin{align} \label{eq:OptimalL_A}
						\frac{\Big|\AA(\xi)M\Big|^2}{\dist_V(\xi)^{2L}} & = (1+\rho^2)^{L-m} \rho^{2l_j - 2L} \Big| (e_1 \circleddash \tilde{\xi})^{\odot l_1} \odot \dots \odot (e_j \circleddash w)^{\odot l_j} \odot \dots (e_d \circleddash \tilde{\xi})^{\odot l_d} \Big|^2.
					\end{align}
					Due to the convergence of $1+\rho^2$ and
					$\Big| (e_1 \circleddash \tilde{\xi})^{\odot l_1} \odot \dots \odot (e_j \circleddash w)^{\odot l_j} \odot \dots \odot (e_d \circleddash \tilde{\xi})^{\odot l_d} \Big|^2$ to non-zero numbers
					for $\delta \to 0$ (and thus $\varepsilon, \rho \to 0$), these are uniformly bounded from below for $\delta > 0$ sufficiently small.  Furthermore, $l_j - L < 0$ and therefore
					$\rho^{2l_j - 2L} \to \infty$ if $l_j \neq L$ or $\rho^{2l_j - 2L} \to 1$ if $l_j = L$, hence also this factor is bounded from below.  In conclusion, also in the second case
					we have the existence of some constant $C = C(\delta,\AA) > 0$ such that
					\begin{align*}
						\frac{\Big|\AA(\xi)M \Big|^2}{\dist_V(\xi)^{2L}} \geq C > 0.
					\end{align*}
					
				Combining both cases for small enough $\delta>0$ shows the claim.  
\end{proof}

		Similarly as for the higher order curl, we deduce a lower order bound for the higher order divergence.
		
		\begin{lem}[The higher order divergence]
			\label{lem:LowerBoundMultipl_div}
			Let $\mathcal{B}(D)$ be given as in \cref{eq:DivergenceOp} and for $l \in \N^d$ with $|l| = m$ let $M = v \otimes e_1^{\odot l_1} \odot \dots \odot e_d^{\odot l_d} \in \R^k \otimes \S{m}$.  We set
				$V = \{\xi \in \R^d: \mathbb{B}(\xi)M = 0\}$ and let $L = m - \min_{j=1,\dots,d} l_j \leq m$.  Then
\begin{align*}
V = \bigcup_{j: l_j \neq 0} \vspan(e_j)^\perp,
\end{align*}				
			and	 there exists a constant $C=C(d,m,v)>0$ such
			that
			\begin{align*}
                          |\mathbb{B}(\frac{\xi}{|\xi|})M|^2 \geq C \frac{\dist_V(\xi)^{2L}}{|\xi|^{2L}}.
			\end{align*}
		\end{lem}

			\begin{proof}
				By virtue of \cref{lem:SymbolDiv} the symbol $\mathbb{B}(\frac{\xi}{|\xi|})M$ is given by $\mathbb{B}(\xi)(v \otimes e_1^{\odot l_1} \odot \dots \odot e_d^{\odot l_d}) = \binom{m}{l}^{-1} \xi^l v$ and therefore
				$V = \bigcup_{j: l_j \neq 0} \vspan(e_j)^\perp$.  Thus the distance to the zero set is given by
				\begin{align*}
					\dist_V(\xi)^2 = \min_{j: l_j \neq 0} |\xi_j|^2.
				\end{align*}
				
				Moreover, we use that for any $\xi \in \mathbb{S}^{d-1}$ there is $k$ such that $|\xi_k|^2 \geq \frac{1}{d}$ and that $L \geq m - l_k$ and therefore
				\begin{align*}
                                  \min_{j: l_j \neq 0} |\xi_j|^{2L} & \leq \min_{j:l_j \neq 0} |\xi_j|^{2(m-l_k)} = \prod_{p: l_p \neq 0, p \neq k} \min_{j:l_j \neq 0} |\xi_j|^{2l_p} \leq \prod_{p: l_p \neq 0, p \neq k} |\xi_p|^{2 l_p} = \frac{\xi^{2l}}{\xi_k^{2l_k}}\\
                                  & \leq d^m \xi^{2l}.
				\end{align*}
				Using the chain of inequalities above together with the fact that $|\mathbb{B}(\xi)M|^2 \geq C(v) \xi^{2l}$, for $\xi \in \mathbb{S}^{d-1}$ there holds
				\begin{align*}
					|\mathbb{B}(\xi)M|^2 \geq C(d,m,v) \dist_V(\xi)^{2L}
				\end{align*}
				and the claim follows.
			\end{proof}

		 As these two lemmata will be used to derive lower scaling bounds in \cref{sec:Applications}, we comment on the choice of $L$.

			\begin{rmk}
			\label{rmk:opt}
				The values of $L$ in \cref{lem:LowerBoundMultipl} and \cref{lem:LowerBoundMultipl_div} are indeed the maximal vanishing order, cf. \cref{defi:index}.  Let us elaborate on this statement:
                                \renewcommand{\labelenumi}{\textbf{\theenumi}} \renewcommand{\theenumi}{(\roman{enumi})}
				\begin{enumerate}
                                \item For the higher order curl, we obtain the characterization of the constant $L$ as follows. By \cref{lem:LowerBoundMultipl}, we immediately infer that $L = \max_{j=1,\dots,d} l_j$ is an upper bound for the maximal vanishing order. To prove that $L$ coincides with the maximal vanishing order, consider $L' < L$. Working as in the proof of \cref{lem:LowerBoundMultipl} with $L'$ in place of $L$, by \cref{eq:OptimalL_A} and the fact that there exists $j \in \{1,\dots,d\}$ such that $l_j -L' > 0$ , we obtain
                                    \begin{align*}
                                      \inf_{\xi \in \mathbb{S}^{d-1} \setminus V} \frac{|\AA(\xi)M|^2}{\dist_V(\xi)^{2L'}} = 0.
                                    \end{align*}
                                \item Similarly as above, for the higher order divergence, \cref{lem:LowerBoundMultipl_div} yields that $L = m - \min_{j=1,\dots,d}l_j$ is an upper bound for the maximal vanishing order.
                                  Now, let
                                  $L' < L$. Then it holds
					\begin{align*}
						\inf_{\xi \in \mathbb{S}^{d-1} \setminus V} \frac{|\mathbb{B}(\xi)M|^2}{\dist_V(\xi)^{2L'}} = 0.
					\end{align*}
					Indeed, choosing $\xi^{(k)}_j = \frac{1}{k}$ for $j \neq j_0$ (where $j_0$ is an index such that $L = m - l_{j_0}$) and $\xi^{(k)}_{j_0} = \sqrt{1-\frac{d-1}{k^2}}$ yields
					\begin{align*}
						\frac{(\xi^{(k)})^{2l}}{\dist_V(\xi^{(k)})^{2L'}} \leq \frac{k^{-2(m-l_{j_0})}(1-\frac{d-1}{k^2})^{l_{j_0}}}{k^{-2(m-l_{j_0}-1)}} = \frac{(1-\frac{d-1}{k^2})^{l_{j_0}}}{k^2} \to 0.
					\end{align*}
				\end{enumerate}
				This indeed proves that our choices of $L$ in \cref{lem:LowerBoundMultipl} and \cref{lem:LowerBoundMultipl_div} correspond to the maximal vanishing orders for these operators.
			\end{rmk}
		
		With the above observations in hand, for the higher order curl and divergence we then obtain the following lower bound estimates.
                
			\begin{proof}[Proof of lower bounds in \cref{thm:scaling_2D_new}]
				By \cref{eq:LowerBoundRRT} and \cref{lem:LowerBoundMultipl} it holds, with $V$ and $L$ as in \cref{lem:LowerBoundMultipl}, that
				\begin{align*}
					E_{el}(\chi;F_\lambda) \geq C \int_{\R^d} |\AA(\frac{\xi}{|\xi|})(A-B) \hat{f}|^2 d\xi \geq C \int_{\R^d} \dist_V(\frac{\xi}{|\xi|})^{2L} |\hat{f}|^2 d\xi,
				\end{align*}
				with $f = (1-\lambda)\chi_A-\lambda\chi_B$ and a constant $C = C(d,L)>0$.  Extending $f$ outside of $\Omega$ by zero, we can apply \cref{thm:LowerScalingUnionVS} and infer that
				\begin{align*}
					E_{\epsilon}(\chi;F_\lambda) \geq C\min\{1-\lambda,\lambda\}^2 \epsilon^{\frac{2L}{2L+1}}.
				\end{align*}
				This concludes the proof.
			\end{proof}

		Analogously, in the setting of the higher order divergence we infer the following lower bounds:

             \begin{lem}
            \label{lem:LowerScalingDiv}
		Let $d,m \in \N, d \geq 2$ and $l \in \N^d$. Let $\Omega \subset \R^d$ be an open, bounded Lipschitz domain. Let $E_{\epsilon}^{\mathcal{B}}(\chi;F)$ be as above in \cref{eq:energy_total_gen}  with the operator $\B(D) = \di^m$ given in \cref{eq:DivergenceOp}. Then the following scaling results hold:
                Let $A-B = v \otimes e_1^{\odot l_1} \odot e_2^{\odot l_2} \odot \cdots \odot e_d^{\odot l_d}$ for some $v \in \R^k$ and such
			that $\sum_{j = 1}^{d} l_{j} = m$ and let $F_{\lambda}:= \lambda A + (1-\lambda ) B$ for some $\lambda \in (0,1)$. Then there exists $C>0$ and $\epsilon_0>0$ (depending on $d, m, \Omega, v$ and $\epsilon_0$ also depending on $\lambda$) such
			that for $L:=m -\min\limits_{j\in\{1,2,\dots,d\}}l_j \leq m$ and for any $\epsilon \in (0,\epsilon_0)$
			\begin{align*}
				C \min\{1-\lambda,\lambda\}^2 \epsilon^{\frac{2L}{2L+1}} \leq \inf\limits_{\chi \in BV(\Omega;\{A,B\})} E_{\epsilon}^{\mathcal{B}}(\chi; F_{\lambda}) .
			\end{align*}
                      \end{lem}
                    As in \cref{thm:scaling_2D_new}, for $d = 2$ this lower bound is optimal for $\lambda = \frac{1}{2}$. In two dimensions the curl and divergence operators only differ by a rotation, therefore this is to be expected once the result for the symmetrized derivative is proved (see \cref{thm:scaling_2D_new}). In \cref{sec:upper} we will further comment on this.

			\begin{proof}[Proof of  \cref{lem:LowerScalingDiv}]
			This is a direct consequence of applying \cref{thm:lower_bound_p} with $L = m - \min\limits_{j=1,\dots,d}l_j$, cf. \cref{lem:LowerBoundMultipl_div}.
			\end{proof}

				\begin{rmk}[Comparison $m$-th order curl and divergence] \label{rmk:comp}
                                  Comparing the results from \cref{lem:LowerBoundMultipl} and \cref{lem:LowerScalingDiv} and noting that the exponents indeed originate from the maximal vanishing order (cf. \cref{rmk:opt}), we observe that since the function $\R \ni t \mapsto \frac{2t}{2t+1} \in \R$ is monotone increasing, the bounds for the higher order curl are, in general, substantially tighter than for the divergence.
                                  Indeed, denoting by $\A(D)$ the higher order curl from \cref{eq:Operator} and by $\mathcal{B}(D)$ the higher order divergence from \cref{eq:DivergenceOp} for $m\in \N$ fixed, $l\in \N^{d}$ with $|l|=m$ and $v\in \R^k$, with the notation from \cref{defi:index}, it follows that
\begin{align*}
L\left[|\mathbb{B}(\xi) v\otimes e_1^{l_1}\odot \cdots \odot e_{d}^{l_d}|^2 \right]:= m - \min\limits_{j=1,\dots,d} l_j \geq \max\limits_{j=1,\dots, d} l_j =: L\left[|\AA(\xi) e_1^{l_1}\odot \cdots \odot e_{d}^{l_d}|^2\right].
\end{align*}
We highlight that this is consistent with the fact that the higher order divergence yields lower bounds for general symbols (cf. \cite[Appendix B]{RRT23}).
\end{rmk}

		\section{Upper Bound Constructions}
		\label{sec:upper}
		
		In this section we provide the arguments for the upper bounds in \cref{thm:scaling_2D_new} in the setting in which $d=2$, $\lambda = \frac{1}{2}$ and with $m\in \N$ general. We emphasize that for $m\in \{1,2\}$ (and general $\lambda \in (0,1)$)
                these results are known (cf. \cite{CC15}). In order to deduce these in the case of higher order tensors we mimic the construction from the setting in \cite{CC15} and
		adapt it correspondingly. We split our discussion into first dealing with the highest possible maximal vanishing order in which $M:=A-B = e_1^{\odot m}$ and then make use of this construction to also infer the result
		for the intermediate cases $M:=A-B = e_1^{\odot l} \odot e_2^{\odot(m-l)}$, $l \in \{1,\dots,m-1\}$.
		
		In what follows, for convenience, we introduce the following notational convention:
		
                \begin{con} \label{rmk:Notation-2D}
                   In order to simplify the notation for $d=2$, we use the symmetries of the tensors and define for $M \in \textup{Sym}(\R^2;m)$
                    \begin{align*}
                      \tilde{M}_0 = M_{1\dots1}, \ \tilde{M}_1 = M_{1\dots12}, \ \tilde{M}_k = M_{1\dots12\dots2},\  \tilde{M}_m = M_{2\dots2},
                    \end{align*}
                    where for $\tilde{M}_k$ there are $k$ many twos and $m-k$ many ones in the index.  Thus, the `new' index counts how many twos appear in the index as by symmetry of $M$ the order of the indices $1$ and $2$ does not matter.  Analogously, we will use this notation for symmetric tensor fields, in particular, for the map $u:\R^2 \to \textup{Sym}(\R^2;m)$ and the potential $v:\R^2 \to \textup{Sym}(\R^2;m-1)$, cf. \cref{eq:potential}.  With this notation it
                    holds for $k =1,\dots,m-1$
                    \begin{align*}
                      \tilde{u}_0 = \p_1 \tilde{v}_1,\ \tilde{u}_k = \frac{m-k}{m} \p_1 \tilde{v}_{k} + \frac{k}{m} \p_2 \tilde{v}_{k-1},\ \tilde{u}_m = \p_2 \tilde{v}_{m-1}.
                    \end{align*}
			Indeed, here we have used that for $(i_1 \dots i_m)=(1\dots12\dots2)$
			\begin{align*}
				\tilde{u}_k &= u_{1\dots12\dots2} = u_{i_1 \dots i_m} =  \frac{1}{m!} \sum_{\tau \in \mathfrak{S}_m} \p_{i_{\tau(1)}} v_{i_{\tau(2)} \dots i_{\tau(m)}} \\
				& = \frac{1}{m!} \sum_{\tau \in \mathfrak{S}_m: i_{\tau(1)} = 1} \p_1 \tilde{v}_{k} + \frac{1}{m!} \sum_{\tau \in \mathfrak{S}_m: i_{\tau(1)} = 2} \p_2 \tilde{v}_{k-1} \\
				& = \frac{|\{\tau \in \mathfrak{S}_m: \tau(1) \in \{1,2,\dots,m-k\}\}|}{m!} \p_1 \tilde{v}_k \\
				& \quad + \frac{|\{\tau \in \mathfrak{S}_m: \tau(1) \in \{m-k+1,m-k+2,\dots,m\}\}|}{m!} \p_2 \tilde{v}_{k-1} \\
				& = \frac{(m-k) (m-1)!}{m!} \p_1 \tilde{v}_k + \frac{k (m-1)!}{m!} \p_2 \tilde{v}_{k-1} = \frac{m-k}{m} \p_1 \tilde{v}_{k} + \frac{k}{m} \p_2 \tilde{v}_{k-m}.
			\end{align*}

		As it will be of relevance in our constructions below, we recall that, for $d=2$, by the antisymmetry properties of $\A(D)u$, it holds $\A(D)u = 0$ if and only if
		$[\A(D)u]_{12\dots12} = 0$.  Moreover it holds
		\begin{align}
                  \label{eq:Operator2d} [\A(D)u]_{12\dots12} = \sum_{k=0}^m \Big( (-1)^k 2^{-m} \binom{m}{k} \p_1^k \p_2^{m-k} \tilde{u}_k \Big),
		\end{align}
		as every time we switch an index one and an index two we multiply by a factor of $(-1)$ and there are exactly $\binom{m}{k}$ possibilities to switch $k$ distinct twos and ones.
              \end{con}

		\subsection{Cell construction}
		\label{sec:cell}
		
	          	In this section, we consider the case $M = e_1^{\odot m}$, which has the largest possible maximal vanishing order.  Following \cite{CC15}, we begin with a unit cell construction (\cref{lem:UnitCell}) in which the higher vanishing order
		of our symbols will be turned into higher order scaling properties.
                This construction requires more careful considerations than for the case of first or second order tensors, as for $m \geq 3$, in general, one cell does not suffice to achieve the desired boundary conditions.
                Next, we will introduce a suitable cut-off procedure and combine these ingredients into a branching construction in
		\cref{sec:highest}.
		The construction will be carried out on the level of a potential, i.e., a map $v: \R^2 \to \textup{Sym}(\R^2;m-1)$, and then we set $u = \Ds v$. To ensure that $\tilde{v}_k$, see \cref{rmk:Notation-2D} for the notation, attains  the desired boundary conditions, we will use a suitable reflection argument.

		\begin{lem}[Unit cell construction] 
                  \label{lem:UnitCell}
                  Let $m \in \N$ and $\A(D)$ be as in \cref{eq:Operator} for $d = 2$.  For $0 < l \le h \leq 1$ let $\omega = (0,2^ml) \times (0,h)$. Let
			$A,B \in \textup{Sym}(\R^2;m)$ be such that $A-B = e_1^{\odot m}$ and let $F = \frac{1}{2} A + \frac{1}{2} B$.  Define for
			$u \in L^2_{loc}(\R^2;\textup{Sym}(\R^2;m)), \chi \in BV(\omega;\{A,B\})$
			\begin{align*}
				E_{el}(u,\chi;\omega) := \int_{\omega} |u-\chi|^2 dx, \ E_{surf}(\chi;\omega) = \int_{\omega} |\nabla \chi|.
			\end{align*}
			Then, there exist a potential $v: \omega \to \textup{Sym}(\R^2;m-1)$, a function $f \in BV(\omega;\{\pm \frac{1}{2}\})$
			and a constant $C = C(m, \Vert \gamma \Vert_{C^m}) > 0$ such that with $u = \Ds v + F\in \mathcal{D}_F$ and $\chi := (A-B)f + F \in BV(\omega;\{A,B\})$ it holds
			\begin{align*}
				E_{el}(u,\chi;\omega) \leq C \frac{l^{2m+1}}{h^{2m-1}}, \ E_{surf}(\chi) \leq C h.
			\end{align*}
			Furthermore it holds that $\p_1 \tilde{v}_0 = f$ and the following boundary conditions are satisfied 
                        
			\begin{align*}
                          v(0,y) &= v(2^ml,y) = 0 & y &\in [0,h], \\
                          v(x,0) & = (-\frac{1}{2}x \chi_{[0,\frac{l}{2})}(x) + \frac{1}{2}(x-l) \chi_{[\frac{l}{2},l]}(x)) e_1^{\odot(m-1)} & x & \in [0,l], \\
                          v(x,0) & = - v(x-2^jl,0) & x & \in [2^jl,2^{j+1}l], \\
                          & & & j \in \{ 0,1,\dots,m-1\}, \\
                          v(x,h) & = - \frac{1}{2} v(2x,0) & x & \in [0,2^{m-1}l],\\
                          v(x,h) & = \frac{1}{2} v(2x-2^{m}l,0) & x & \in [2^{m-1}l,2^ml], \\
                          \tilde{v}_k(x,0) & = \tilde{v}_k(x,h) =  0 & x & \in [0,2^ml], \\
                          && &k \in \{1,\dots,m-1\}.
                        \end{align*}
                        
		\end{lem}
		
		\begin{proof}
			We adopt the ideas from the upper bound construction in \cite{CC15} to our context.  To achieve this, we first reduce to the setting of $F = 0$ by subtracting the boundary data:
			\begin{align*}
				E_{el}(u,\chi;\omega) = \int_{\omega} |(u-F) - (\chi-F)|^2 dx.
			\end{align*}
			Hence, without loss of generality we may assume $F = 0$ and therefore $\frac{1}{2}A+\frac{1}{2}B=0$.
			
			Using this we introduce a function $f:\omega \to \{\pm \frac{1}{2}\}$ such that the phase indicator reads $\chi = (A-B) f = f e_1^{\odot m}$.  Plugging this into the elastic energy and recalling \cref{rmk:Notation-2D}, we obtain
					\begin{align*}
                                          E_{el}(u,\chi;\omega) & \leq C(m) \int_\omega |\tilde{u}_0 -f|^2 + \sum_{j=1}^m |\tilde{u}_j|^2 dx \\
                                          & = C(m) \int_\omega |\tilde{u}_0 - f|^2 + \sum_{j=1}^{m-1} |\tilde{u}_j|^2 + |\tilde{u}_m|^2 dx.
			\end{align*}
			
			Let us next outline the idea of constructing the tensor $u$: To this end, we will first fix $\tilde{u}_0$ using the construction from \cite{CC15}, which leads to $|\tilde{u}_0 - f| =
			0$. Iteratively, we will then define the remaining $\tilde{u}_k$ by setting them to zero except for $\tilde{u}_m$.  This will lead to an energy bound of the form
			\begin{align*}
				E_{el}(u,\chi;\omega) \leq C \int_\omega |\tilde{u}_m|^2 dx,
			\end{align*}
			where $\tilde{u}_m$ is determined by all the other components through the constraint $\A(D) u = 0$.
                        In comparison to the cases of $m=1,2$, i.e., of the curl and curl curl operator, for higher $m$ the argument to achieve the boundary condition is more involved.
                        To ensure this, we rely on an (iterative) reflection-type argument. We split the proof into several substeps carrying this out successively.

                      \textbf{Step 1: Preliminary definitions and outline of the strategy.}
			In order to implement the outlined strategy, we first define a monotone function $\gamma \in C^{\infty}(\R;[0,1])$ such that 
			\begin{align*}
				\gamma(t) = 0 \text{ for } t \leq \delta, \ \gamma(t) = 1 \text{ for } t \geq 1-\delta \mbox{ for some } \delta \in (0,\frac{1}{4}).
			\end{align*}
                        We start by giving the arguments in a smaller cell $\omega_0 := [0,l] \times [0,h]$ which we will reflect suitably to achieve zero boundary values for the potential.
			We split $\omega_0$ into three subregions given by (cf. \cite{CC15})
			\begin{align*}
				\omega_1 & := \Big\{(x,y) \in \omega_0 : x \in \big[0 , \gamma(\frac{y}{h})\frac{l}{4}\big)\Big\},\\
				\omega_2 & := \Big\{(x,y) \in \omega_0 : x \in \big[ \gamma(\frac{y}{h}) \frac{l}{4}, \frac{l}{2}+\gamma(\frac{y}{h}) \frac{l}{4} \big) \Big\},\\
				\omega_3 & := \Big\{(x,y) \in \omega_0 : x \in \big[ \frac{l}{2}+\gamma(\frac{y}{h}) \frac{l}{4}, l \big]\Big\}.
			\end{align*}
			We define
			\begin{align*}
				f(x,y) := \begin{cases}\frac{1}{2} & (x,y) \in \omega_1 \cup \omega_3, \\ -\frac{1}{2} & (x,y) \in \omega_2. \end{cases}
			\end{align*}
                        This is illustrated in \cref{fig:branching}.
                        
			Our next goal is to set $\tilde{u}_0 = f$ and then to iteratively fix the other components such that $\tilde{u}_k = 0$ for $k\in \{1,\dots,m-1\}$.  To this end, we use a potential
			$v: \omega \to \textup{Sym}(\R^2;m-1)$ and define $\tilde{v}_0$ by integration of $\tilde{u}_0 = \p_1 \tilde{v}_0$ in $\omega_0$. We then fix $\tilde{v}_0$ by a reflection-type argument in the remainder of $\omega = [0,2^ml]\times [0,h]$.
                        More precisely, we define
			\begin{align*}
				\tilde{v}_0(x,y) := \begin{cases}
                                                          \frac{1}{2} x & (x,y) \in \omega_1, \\
                                                          - \frac{1}{2} x + \gamma(\frac{y}{h}) \frac{l}{4} & (x,y) \in \omega_2, \\
                                                          \frac{1}{2} (x-l) & (x,y) \in \omega_3.
				\end{cases}
			\end{align*}
                        Therefore we have in $\omega_0$
			\begin{align*}
				\p_1 \tilde{v}_0(x,y) & = \begin{cases} \frac{1}{2} & (x,y) \in \omega_1 \cup \omega_3, \\ - \frac{1}{2} & (x,y) \in \omega_2, \end{cases}\\
				\p_2 \tilde{v}_0(x,y) & = \begin{cases} 0 & (x,y) \in \omega_1 \cup \omega_3, \\ \gamma'(\frac{y}{h}) \frac{l}{4h} & (x,y) \in \omega_2. \end{cases}
			\end{align*}

			Implementing the above outlined idea, we seek to define the $\tilde{v}_k$ by setting $\tilde{u}_k = 0$ for $k \in \{1,\dots,m-1\}$. As a consequence, we iteratively solve the equations
			\begin{align*} 
			\begin{cases}
				\displaystyle\vspace{1.5ex}
				\frac{1}{m} \p_2 \tilde{v}_0 + \frac{m-1}{m} \p_1 \tilde{v}_1 = 0, \\
				\displaystyle\vspace{1.5ex}
				\frac{k}{m} \p_2 \tilde{v}_{k-1} + \frac{m-k}{m} \p_1 \tilde{v}_k = 0, & k \in \{2,\dots,m-2\}, \\
				\displaystyle
				\frac{m-1}{m} \p_2 \tilde{v}_{m-2} + \frac{1}{m} \p_1 \tilde{v}_{m-1} = 0,
			\end{cases}
			\end{align*}
			with the boundary condition $\tilde{v}_k(0,y) = 0$.  The function $\tilde{v}_k$ is then defined in terms of $\tilde{v}_{k-1}$ by
			\begin{align}
			\label{eq:pot_iterative}
				\tilde{v}_k(x,y) := - \frac{k}{m-k} \int_{0}^x \p_2 \tilde{v}_{k-1}(t,y) dt.
			\end{align}
Notice that $\tilde{v}_0(l,y) = 0$, but for larger $k\ge1$ this in general fails. Thus, we use appropriate reflection arguments in order to attain a zero right-boundary condition.
                          Indeed, using the fact that $-\tilde{v}_0(x,y)$ fulfills $\p_1 (-\tilde{v}_0) \in \{\pm \frac{1}{2}\}$, we implement a reflection-type argument in the cell $[0,2l] \times [0,h]$, i.e., for $x \in [l,2l]$, we set
                          \begin{align*}
                            \tilde{v}_0(x,y) = -\tilde{v}_0(x-l,y).
                          \end{align*}
                          This, by the iterative definition \cref{eq:pot_iterative}, immediately implies that $\tilde v_1(2l,y)=0$.
By exploiting yet another reflaction argument, i.e., for $x \in [2l,4l]$
                          \begin{align*}
                            \tilde{v}_0(x,y) = -\tilde{v}_0(x-2l,y).
                          \end{align*}
and the fact that $\tilde v_1(2l,y)=0$, this "reflection" property carries over to $\tilde{v}_1$ in the sense that for $x \in [2l,4l]$, we then have $\tilde{v}_1(x,y) = -\tilde{v}_1(x-2l,y)$.
                          By this, it then also follows that
                          \begin{align*}
                            \int_0^{4l} \tilde{v}_1(x,y) dx = 0,
                          \end{align*}
which in turn ensures that $\tilde v_2(4l,y)=0$.
                          As this still does not suffice to ensure zero boundary conditions for $\tilde{v}_k$ for $k \geq 3$, we iterate further this reflection in the next steps.
                          
                            \textbf{Step 2: Definition of $\tilde{v}_0: \omega \to \R$.}
                          Building on the outlined reflection idea, we define $\tilde{v}_0: \omega \to \R$ to be given by
                          \begin{align*}
                            \tilde{v}_0(x,y) & :=
                            \begin{cases}
                              \frac{1}{2} x & (x,y) \in \omega_1, \\
                              -\frac{1}{2}x + \gamma(\frac{y}{h}) \frac{l}{4} & (x,y) \in \omega_2, \\
                              \frac{1}{2}(x-l) & (x,y) \in \omega_3,
                            \end{cases}
                                               \quad \text{for } x \in [0,l], \\
                            \tilde{v}_0(x,y) & := - \tilde{v}_0(x-2^jl,y), \quad \text{for } x \in [2^jl,2^{j+1}l], j\in \{0,1,\dots,m-1\}.
                          \end{align*}
                          As $\tilde{v}_0(0,y) = \tilde{v}_0(l,y)$ this function is continuous and well-defined and $\tilde{v}_0(2^j l,y) = 0$ for all $j\in\{1,\dots,m\}$.
                          Furthermore, it holds $\p_1 \tilde{v}_0(x,y) \in \{\pm \frac{1}{2} \}$ and for $y \in [0,\delta h]$
                          \begin{align*}
                            \tilde{v}_0(x,y) &= \begin{cases} - \frac{1}{2} x &x \in [0,\frac{l}{2}), \\ \frac{1}{2}(x-l) & x \in [\frac{l}{2},l),\\
\frac{1}{2}(x-l) & x\in [l, \frac{3}{2}l),\\
-\frac{1}{2}(x- 2l) & x\in [ \frac{3}{2}l, 2l].                           
                             \end{cases} 
                          \end{align*}
                          For $y \in [(1-\delta)h,h]$ and $x \in [0,l]$ we have
                          \begin{align*}
                            \tilde{v}_0(x,y) &= \begin{cases} \frac{1}{2} x & x \in [0,\frac{l}{4}), \\ -\frac{1}{2}(x-\frac{l}{2}) & x \in [\frac{l}{4},\frac{3l}{4}), \\ \frac{1}{2}(x-l) & x \in [\frac{3l}{4},l], \end{cases}
                          \end{align*}
                          which, for $x\in [0,l]$, can be written as
                          \begin{align*}
                            \tilde{v}_0(x,y) = - \frac{1}{2} \tilde{v}_0(2x,0).
                          \end{align*}
                          Iteratively, by using that $\tilde{v}_0(x,y) = -\tilde{v}_0(x-2^jl,y)$ for $x \in [2^jl,2^{j+1}l]$, $j\in\{0,\dots,m-1\}$, this carries on, and hence
                          \begin{align*}
                            \tilde{v}_0(x,h) &= - \frac{1}{2} \tilde{v}_0(2x,0), & x &\in [0,2^{m-1}l], \\
                            \tilde{v}_0(x,h) &= \frac{1}{2} \tilde{v}_0(2x-2^ml,0), & x &\in [2^{m-1}l,2^ml].
                          \end{align*}
                        
In the end, we also set $f:\omega\to\{\pm\frac{1}{2}\}$ as $f(x,y)=\p_1\tilde v_0(x,y)$.

                          \textbf{Step 3: Definition of $\tilde{v}_k:\omega \to \R$.}
                          With $\tilde{v}_0$ defined on $\omega$, we iteratively define $\tilde{v}_1$, $\tilde{v}_2,\dots,$ $\tilde{v}_{m-1}$ on $\omega$ by using the equation for the potential.
                          We define iteratively for $k\in\{1,2,\dots,m-1\}$
                          \begin{align*}
                            \tilde{v}_1(x,y) := -\frac{1}{m-1}\int_0^x \p_2 \tilde{v}_0(t,y) dt, \quad \tilde{v}_k(x,y) := - \frac{k}{m-k} \int_0^x \p_2 \tilde{v}_{k-1}(t,y) dt.
                          \end{align*}
                          As shown in Step 4 below these functions are Lipschitz continuous. 
                          By this definition, it holds for a.e. $(x,y) \in \omega$
                          \begin{align} \label{eq:DiffEqPotential}
                            \frac{k}{m} \p_2 \tilde{v}_{k-1}(x,y) + \frac{m-k}{m} \p_1 \tilde{v}_{k}(x,y) = 0, \quad k\in\{1,2,\dots,m-1\}.
                          \end{align}
                          We now claim that for $k \in \{1,\dots,m-1\}$ the following properties are satisfied
                          \begin{align}
                            \int_0^{2^{k+1}l}\tilde{v}_{k}(t,y) dt &=0, \label{eq:Mean_vk}\\
                            \tilde{v}_k(x,y) &= - \tilde{v}_k(x-2^{j+k}l,y),& \text{for } x &\in [2^{j+k}l,2^{j+1+k}l], \label{eq:Extension_vk}\\
                            &&j&\in \{0,1,\dots,m-1-k\}, \nonumber \\
                            \tilde{v}_k(2^{j+k}l,y) & = 0 = \tilde{v}_k(0,y), &\text{for } j &\in \{ 0,1,\dots,m-k\}. \label{eq:Zero_vk}
                          \end{align}
                        
Properties \cref{eq:Extension_vk,eq:Zero_vk} will be proved by finite induction in Steps 5 and 6 below.  The property \cref{eq:Mean_vk} then directly follows from \cref{eq:Extension_vk}.

                        \textbf{Step 4: Regularity of $\tilde{v}_k$.}
                        We claim that each $\tilde{v}_k$ is Lipschitz continuous. We will first discuss this only in $\omega_0:=[0,l]\times[0,h]$ and will then comment on how this immediately implies Lipschitz continuity on the full cell $\omega$.

Let $G:\omega_0\to\R$ be given in the form $G(x,y) = G_1(x,y) \chi_{\omega_1}(x,y) + G_2(x,y) \chi_{\omega_2} (x,y) + G_3(x,y) \chi_{\omega_3}(x,y)$, with $G_j$ being Lipschitz functions.
Notice that both $f$ and $\p_2\tilde v_0$ have this form.

\smallskip

\emph{Claim 1:}
Let $g(x,y)=\int_0^x G(t,y)dt$, then $g$ is Lipschitz and $\|\nabla g\|_{L^\infty(\omega_0)}\le C(\|G\|_{L^\infty(\omega_0)}+\sum_{j=1}^3\|\p_2 G_j\|_{L^\infty(\omega_j)})$, where $C>0$ depends on $l, \frac{l}{h}, \|\gamma'\|_{L^{\infty}}$.

\smallskip

\emph{Indeed:} for every $(x_1,y_1),(x_2,y_2)\in\omega_0$ we have
\begin{align*}
|g(x_2,y_2)-g(x_1,y_1)| &= \Big|\int_0^{x_2}G(t,y_2)dt-\int_0^{x_1}G(t,y_1)dt\Big| \\
&\le \Big|\int_0^{x_2}G(t,y_2)-G(t,y_1)dt+\int_{x_1}^{x_2}G(t,y_1)dt\Big|.
\end{align*}
For a.e. $t\in(0,l)$, $G(t,\cdot)\in BV((0,h))$ and by the representation theorem of one-dimensional BV functions (see e.g.\ \cite[Theorem 3.28 and Corollary 3.33]{AFP2000})
\begin{align*}
\int_0^{x_2}G(t,y_2)-G(t,y_1)dt &= \int_0^{x_2}\Big(\int_{y_1}^{y_2}\p_2 G(t,s)ds+\int_{y_1}^{y_2}[G](t,y_t) d\delta_{y_t}(s)\Big) dt,
\end{align*}
where $\p_2 G$ denotes the absolutely continuous part of the derivative of $G$, $[G](t,y_t)$ is the amplitude of the jump of $G(t,\cdot)$ in $y_t$ and $\{(t,y_t)\}=(\cup_{j=1}^3\p\omega_j)\cap(\{t\}\times(y_1,y_2))$ which is unique for fixed $t \in (0,l)$ by definition of the sets $\omega_j$ (see \cref{fig:branching}).
Without loss of generality we assume $y_1 \leq y_2$ and obtain that the last term above can be controlled as 
\begin{align*}
  \Big|\int_0^{x_2} \int_{y_1}^{y_2}[G](t,y_t) d\delta_{y_t}(s) dt\Big| &= \left| \int_0^{x_2} [G](t,y_t) \chi_{(y_1,y_2)}(y_t) dt \right| \\
  &\le 4\|G\|_{L^\infty(\omega_0)}\frac{l}{4}\Big|\gamma\Big(\frac{y_2}{h}\Big)-\gamma\Big(\frac{y_1}{h}\Big)\Big|.
\end{align*}
Combining the three formulas above, by regularity of $\gamma$, we obtain
\begin{align*}
|g(x_2,y_2)-g(x_1,y_1)| &\le \Big|\int_0^{x_2}\int_{y_1}^{y_2}\p_2 G(t,s)dsdt\Big|+ \Big| \int_0^{x_2} [G](t,y_t) \chi_{(y_1,y_2)}(y_t) dt\Big| \\
&\quad +\Big|\int_{x_1}^{x_2}G(t,y_1)dt\Big| \\
&\le l|y_2-y_1| \sum_{j=1}^3\|\p_2 G_j\|_{L^{\infty}(\omega_j)}+ \frac{l}{h}\|\gamma'\|_{L^\infty}\|G\|_{L^{\infty}(\omega_0)}|y_2-y_1| \\
&\quad +|x_2-x_1|\|G\|_{L^\infty(\omega_0)},
\end{align*}
which yields the claim by recalling the condition $0<l\le h\le1$.

\smallskip

Since we seek to iterate this, we need to prove that $\p_2 g$ has the same structure as $G$, if we start with $G_j$ sufficiently regular.

\smallskip

\emph{Claim 2:} Let $G$ be as above with $G_j\in C^\infty(\overline{\omega_j})$. Then there exist functions $\tilde g_j\in C^\infty(\overline{\omega_j})$, $j \in \{1,2,3\}$, such that $\p_2 g(x,y)=\tilde g_1(x,y)\chi_{\omega_1}(x,y)+\tilde g_2(x,y)\chi_{\omega_2}(x,y)+\tilde g_3(x,y)\chi_{\omega_3}(x,y)$.

\smallskip

\emph{Indeed:} we have
\begin{align*}
g(x,y)=\int_0^x G(t,y)dt = \sum_{j=1}^3 \int_0^x G_j(t,y)\chi_{\omega_j}(t,y)dt= \sum_{j=1}^3 \int_{((0,x)\times\{y\})\cap\omega_j} G_j(t,y)dt.
\end{align*}
By definition $((0,x)\times\{y\})\cap\omega_j=(\min\{a_j(y), x\}, \min\{b_j(y), x\})\times\{y\}$ for  $a_1=0$, $b_1(y)=a_2(y)=\frac{l}{4}\gamma(\frac{y}{h})$, $b_2(y)=a_3(y)=\frac{l}{2}+\frac{l}{4}\gamma(\frac{y}{h})$, $b_3(y)=l$. Note that $a_j, b_j\in C^\infty$.
Denoting
\begin{align*}
g_j(x,y)=\int_{a_j(y)}^x G_j(t,y)dt,
\end{align*}
which are $C^\infty$ functions in $y$, we get
\begin{align*}
g(x,y)=\sum_{j=1}^3 \left(g_j(x,y)\chi_{\omega_j}(x,y)+g_j(b_j(y),y)\chi_{\tilde\omega_j}(x,y) \right),
\end{align*}
where $\tilde \omega_1:=\omega_2\cup\omega_3$, $\tilde\omega_2:=\omega_3$ and $\tilde\omega_3:=\emptyset$.
The distributional derivative of $g$ in the $x_2$ direction can be computed to read
\begin{align*}
\p_2 g(x,y) &= \sum_{j=1}^3(\p_2g_j(x,y)\chi_{\omega_j}(x,y)-g_j(b_{j}(y),y) b_j'(y)\delta_{b_{j}(y)}(x)) \\
&\qquad\qquad +\Big(\left(\frac{d}{dy}g_j(b_j(y),y) \right)\chi_{\tilde{\omega}_j}(x,y)+g_j(b_{j}(y),y) b_j'(y)\delta_{b_{j}(y)}(x))\Big) \\
&= \sum_{j=1}^3\Big(\p_2g_j(x,y) \chi_{\omega_j}(x,y)+\left(\frac{d}{dy}g_j(b_j(y),y) \right) \chi_{\tilde{\omega}_j}(x,y)\Big).
\end{align*}
Here we used the fact that $g_j(a_j(y),y) = 0$ for $j=1,2,3$ and, with slight abuse of notation, we used $ b_j'(y)\delta_{b_{j}(y)}(x)$ to denote the distribution 
\begin{align*}
T_{b_j'(y) \delta_{b_{j}(y)}(x) }[\varphi] := \int\limits_{\R} b_j'(y) \varphi(b_j(y),y) dy \mbox{ for all } \varphi \in C_c^{\infty}(\omega_0).
\end{align*}
Since $\tilde{\omega}_j$ can be written as suitable unions of $\omega_k$, $k\in \{1,2,3\}$, the claim follows by recalling the higher regularity of $g_j$ and $b_j$ (which also provides an additional argument for Claim 1).

\smallskip

Applying iteratively the two claims above, we conclude the Lipschitz regularity of $\tilde{v}_k$.
We highlight that the $L^\infty$ bounds below are only a priori bounds and will be improved later in Step 8.

\smallskip

\emph{Claim 3:} For every $k\in \{0,1,\dots,m-1\}$, $\tilde v_k:\omega\to\R$ is Lipschitz continuous and $\|\tilde v_k\|_{L^\infty(\omega)}$, $\|\nabla\tilde v_k\|_{L^\infty(\omega)}\leq C(m,\Vert \gamma \Vert_{C^m},l,h)$.

\smallskip

\emph{Indeed:} we reason by finite induction. By Claim 1, applied to $G=f$, $\tilde v_0$ complies with Claim 3, and by Claim 2 it has the desired structure.

Assume now, by induction, that $\tilde v_{k-1}$ complies with Claim 3 and is such that $\p_2 \tilde{v}_{k-1}= \sum_{j=1}^3 G_j^{(k-1)}\chi_{\omega_j}$.
Then applying Claim 1 with $G^{(k-1)}_j=-\frac{k}{m-k}\p_2\tilde v_{k-1} \chi_{\omega_j}$, we obtain that $\tilde v_k$ complies with Claim 3 and has the desired structure.

In the end we notice that, if $(x,y)\in\omega_0+(l,0)$, then
\begin{align*}
\tilde v_k(x,y)=\tilde v_k(l,y)-\frac{k}{m-k}\int_l^x\p_2\tilde v_{k-1}(t,y)dt,
\end{align*}
from which we immediately infer that $\tilde v_k$ is Lipschitz also in $[0,2l]\times[0,h]$ and, again by finite induction, we obtain the claim in the full domain $\omega$.

                          \textbf{Step 5: Induction basis; properties for $\tilde{v}_1$.}
                           Properties \cref{eq:Mean_vk,eq:Extension_vk,eq:Zero_vk} are shown by induction, with the induction basis being given by $k=1$.
                          We note that $\p_2 \tilde{v}_0(x,y) = \gamma'(\frac{y}{h}) \frac{l}{4h} \chi_{\omega_2}(x,y)$ for $x \in [0,l]$.
                          Spelling out the definition $\tilde{v}_1(x,y)= - \frac{1}{m-1} \int_0^x \p_2 \tilde{v}_0(t,y) dt$ yields for $x \in [0,2l]$
                          \begin{align*}
                            \tilde{v}_1(x,y) = \begin{cases} 0 & (x,y) \in \omega_1, \\ \gamma'(\frac{y}{h}) \frac{l}{4h} (x-\gamma(\frac{y}{h}) \frac{l}{4}) & (x,y) \in \omega_2, \\ \gamma'(\frac{y}{h}) \frac{l^2}{8h} & (x,y) \in \omega_3 \cup \omega_1 + (l,0), \\
                                                 - \gamma'(\frac{y}{h}) \frac{l}{4h}(x-\frac{3l}{2}-\gamma(\frac{y}{h})\frac{l}{4}) & (x,y) \in \omega_2 + (l,0), \\
                                                 0 & (x,y) \in \omega_3+(l,0).
                                               \end{cases}
                          \end{align*}
                          Thus, indeed it holds $\tilde{v}_1(0,y) = \tilde{v}_1(2l,y) = 0$.

                          Let us show \cref{eq:Extension_vk,eq:Zero_vk} for $k=1$, i.e., $\tilde{v}_1(x,y) = - \tilde{v}_1(x-2^j 2l,y)$ for $x \in [2^{j+1}l,2^{j+2}l]$ and $\tilde{v}_1(2^{j+1}l,y)=0$ with $j\in\{0,\dots,m-2\}$:
                          To this end, let $x \in [2l,4l]$, then
                          \begin{align*}
                            \tilde{v}_1(x,y) & = -\frac{1}{m-1} \int_0^x \p_2 \tilde{v}_0(t,y) dt = \tilde{v}_1(2l,y) - \frac{1}{m-1}\int_{2l}^x \p_2 \tilde{v}_0(t,y) dt \\
                            &=-\frac{1}{m-1} \int_{2l}^x - \p_2 \tilde{v}_0(t-2l,y) dt = \frac{1}{m-1} \int_0^{x-2l} \p_2\tilde{v}_0(t,y) dt = - \tilde{v}_1(x-2l,y).
                          \end{align*}
                          This also implies $\tilde{v}_1(4l,y)=-\tilde{v}_1(2l,y)=0$.
                          Therefore, assuming by induction that for some $j_0 \in \{0,1,\dots,m-2\}$ it holds $\tilde{v}_1(2^{j+1}l,y)=0$ for $j\in\{0,1,\dots,j_0\}$ and also \cref{eq:Extension_vk} holds for $x\in[0,2^{j_0+1}l]$, then for $x \in [2^{j_0+1}l,2^{j_0+2}l]$ we have
                          \begin{align*}
                            \tilde{v}_1(x,y) &= - \frac{1}{m-1} \int_0^x \p_2 \tilde{v}_0(t,y) dt = \tilde{v}_1(2^{j_0+1}l,y) - \frac{1}{m-1} \int_{2^{j_0+1}l}^{x} \p_2 \tilde{v}_0(t,y) dt \\
                                             & = \frac{1}{m-1} \int_{2^{j_0+1}l}^{x} \p_2 \tilde{v}_0(t-2^{j_0+1}l,y) dt = \frac{1}{m-1} \int_0^{x-2^{j_0+1}l} \p_2 \tilde{v}_0(t,y) dt \\
                            & = -\tilde{v}_1(x-2^{j_0+1}l,y).
                          \end{align*}
                          This shows that \cref{eq:Extension_vk} holds also in $[2^{j_0+1}l,2^{j_0+2}l]$ and, moreover, as a consequence we have $\tilde{v}_{1}(2^{j_0+2}l,y) = -\tilde{v}_{1}(2^{j_0+1}l,y) = 0$.
                     
                          \textbf{Step 6: Induction step; properties for $\tilde{v}_k$.}
                          Assuming now that \cref{eq:Mean_vk,eq:Extension_vk,eq:Zero_vk} hold for $1,\dots,k-1$, we seek to prove that they also hold for $k$.
                          By definition, we have
                          \begin{align*}
                            \tilde{v}_k(x,y) = -\frac{k}{m-k} \int_0^x \p_2 \tilde{v}_{k-1}(t,y) dt.
                          \end{align*}
                          As $\int_0^{2^kl} \tilde{v}_{k-1}(t,y) dt = 0$ by \cref{eq:Mean_vk} for $k-1$, we deduce $\tilde{v}_k(2^kl,y) = 0 = \tilde{v}_k(0,y)$.                          
                        
                          As for $k=1$, we next deduce \cref{eq:Extension_vk,eq:Zero_vk} inductively.
                          We start by giving the argument for \cref{eq:Extension_vk} for $j=0$: Let $x \in [2^kl,2^{k+1}l]$, then
                          \begin{align*}
                            \tilde{v}_k(x,y) &= -\frac{k}{m-k} \int_0^x \p_2 \tilde{v}_{k-1}(t,y) dt = \tilde{v}_k(2^kl,y) - \frac{k}{m-k} \int_{2^kl}^x - \p_2 \tilde{v}_{k-1}(t-2^kl,y)dt \\
                            & = \frac{k}{m-k} \int_0^{x-2^kl} \p_2 \tilde{v}_{k-1}(t,y) dt = - \tilde{v}_k(x-2^kl,y).
                          \end{align*}
                          Again, as in the case $k=1$, this can be continued iteratively. Assume $\tilde{v}_k(2^{j+k}l,y)=0$ for $j\in\{0,1,\dots,j_0\}$ and also that \cref{eq:Extension_vk} holds in $[0,2^{j_0+k}l]$, then we have for $x \in [2^{j_0+k}l,2^{j_0+k+1}l]$
                          \begin{align*}
                            \tilde{v}_k(x,y) &= \tilde{v}_k(2^{j_0+k}l,y) - \frac{k}{m-k} \int_{2^{j_0+k}l}^x \p_2 \tilde{v}_{k-1}(t,y) dt \\
                            &= - \frac{k}{m-k} \int_{2^{j_0+k}l}^x - \p_2 \tilde{v}_{k-1}(t-2^{j_0+k}l,y) dt \\
                            & = \frac{k}{m-k} \int_{0}^{x-2^{j_0+k}l} \p_2 \tilde{v}_{k-1}(t,y) dt = - \tilde{v}_k(x-2^{j_0+k}l,y).
                          \end{align*}
                          Thus \cref{eq:Extension_vk} holds in $[0,2^{j_0+1+k}l]$ and, moreover, by an application of \cref{eq:Extension_vk}, we get $\tilde{v}_k(2^{j_0+1+k}l,y) = - \tilde{v}_k(2^{j_0+k}l,y) = 0$, i.e., \cref{eq:Zero_vk} also holds for $j=j_0+1$.
                          By induction, and the above arguments, we have shown \cref{eq:Mean_vk,eq:Extension_vk,eq:Zero_vk} for $k\in\{1,\dots,m-1\}$, where we stop at order $m-1$ to obtain $v: \omega \to \textup{Sym}(\R^2;m-1)$.

                          \textbf{Step 7: Conclusion of the construction.}
                          Furthermore, as $\gamma'(t) = 0$ for $t \in [0,\delta] \cup [1-\delta,1]$, we have
                          \begin{align*}
                            \p_2 \tilde{v}_0(x,y) = 0, \quad \text{for } y \in [0,\delta h] \cup [(1-\delta)h,h],
                          \end{align*}
                          and therefore for $k\in\{1,\dots,m-1\}$ it holds
                          \begin{align*}
                            \tilde{v}_k(x,y) = 0, \quad \text{for } y \in [0,\delta h] \cup [(1-\delta)h, h].
                          \end{align*}
                          Considering \cref{eq:Zero_vk} for $k \in \{ 1,\dots,m-1\}$ and $j=m-k$, we see that $\tilde{v}_{k}(2^{j+k}l,y) = \tilde{v}_k(2^ml,y) = 0$. In particular, due to the size of $\omega$, it holds $\tilde{v}_{m-1}(2^m l,y) = 0$.

                          \textbf{Step 8: Energy estimates.}
                          In conclusion, we have $v: \omega \to \textup{Sym}(\R^2;m-1)$, with $v(x,y) = 0$ for $x\in\{0,2^ml\}$ and, moreover, $\tilde{v}_k(x,y) = 0$ for $k\in \{1,\dots,m-1\}$ at $y=0$ and $y=h$.
                          Setting now $u = \Ds v$ it holds
                          \begin{align*}
                            u(x,y) =  f(x,y) e_1^{\odot m} + \p_2 \tilde{v}_{m-1}(x,y) e_2^{\odot m}, \quad (x,y) \in \omega,
                          \end{align*}
                            by the iterative definition of $\tilde{v}_k$.
                            Choosing $\chi(x,y) := f(x,y) e_1^{\odot m} \in BV(\omega;\{A,B\})$ we therefore have
                            \begin{align}\label{eq:IntermediateBoundElastic}
                              E_{el}(u,\chi;\omega) \leq C(m) \int_\omega |\p_2 \tilde{v}_{m-1}|^2 dx.
                            \end{align}

                           It remains to provide the bound for $\p_2 \tilde{v}_{m-1}$.
To this end, we consider $\tilde w_k:[0,2^m]\times[0,1]\to\R$ defined as $\tilde w_k(x',y'):=\frac{h^k}{l^{k+1}}\tilde v_k(lx',hy')$.
Notice that $w$ coincides with $v$ when $l=h=1$; it is indeed just a rescaled version of $v$.
Since $w$ is a particular choice of $v$ (when $l=h=1$), it has the same structure, i.e., it solves the same system (with rescaled $f$).
Indeed, for $k\in \{1,\dots,m-2\}$ we obtain
\begin{align*}
&k\p_2 \tilde w_{k-1}(x',y')+(m-k)\p_1 \tilde w_k(x',y') \\
&= k\p_{y'} \Big(\frac{h^{k-1}}{l^{k}}\tilde v_{k-1}(lx',hy')\Big)+(m-k)\p_{x'} \Big(\frac{h^k}{l^{k+1}} \tilde v_k(lx',hy')\Big) \\
&= (k\p_2 \tilde v_{k-1}(lx',hy')+(m-k)\p_1\tilde v_k(lx',hy'))\frac{h^k}{l^{k}}=0.
\end{align*}
We also have
\begin{align*}
\p_1\tilde w_0(x',y')=\p_{x'}\frac{1}{l}\tilde v_{0}(lx',hy')=\p_1\tilde v_{0}(lx',hy')=f(lx',hy').
\end{align*}
By Claim 3 in Step 4, applied to $l=h=1$, we have $\|\tilde w_k\|_{L^\infty},\|\nabla \tilde w_k\|_{L^\infty}\leq C(m,\Vert \gamma \Vert_{C^m})$ for every $k \in \{0,1,\dots, m-1\}$.
Finally,
\begin{align*}
|\p_2\tilde v_k(x,y)|=\frac{l^{k+1}}{{h^k}}|\p_y\tilde w(x/l,y/h)|=\frac{l^{k+1}}{h^{k+1}}|\p_2\tilde w(x/l,y/h)|\leq C(m,\Vert \gamma \Vert_{C^m}) \frac{l^{k+1}}{h^{k+1}}.
\end{align*}

  Plugging this for $k=m-1$ into the above estimate \cref{eq:IntermediateBoundElastic}, yields
 \begin{align*}
    E_{el}(u,\chi;\omega) \leq C(m) \int_\omega \frac{l^{2m}}{h^{2m}} d(x,y) = C(m) 2^m \frac{l^{2m+1}}{h^{2m-1}}.
  \end{align*}

  To bound the surface energy, we note that $f=\p_1\tilde{v}_0$ has only two interfaces in the interior of each $([0,l] + jl) \times [0,h]$ cell.
  Thus, we can bound the surface energy, using $l < h$, by
                            \begin{align*}
                              E_{surf}(\chi;\omega) & = \int_\omega |\nabla \chi| \leq C 2^m \left(2 \int_0^h \sqrt{1+\left( \gamma'(\frac{y}{h}) \frac{l}{4h}\right)^2} dy + 2h \right) \\
                              & \leq C(m) \Vert \gamma' \Vert_\infty h.
                            \end{align*}
                          
		\end{proof}
		
		\begin{rmk} \label{rmk:UnitCellRemark} Let us comment on three technical aspects of the previous construction:
			\begin{itemize}
				\item \emph{The function $\gamma$.} A possible choice for the function $\gamma$ as in the proof of \cref{lem:UnitCell} is given by
				\begin{align*}
					\gamma(t) = \frac{h(t-\delta)}{h(1-\delta-t)+h(t-\delta)},
				\end{align*}
				where $h: \R \to \R$ is the smooth function defined by
				\begin{align*}
					h(t) = \begin{cases} e^{-\frac{1}{t}} & t>0,\\ 0 & t \leq 0. \end{cases}
				\end{align*}
				
				Moreover, we emphasize that for $A-B $ as above, it would be sufficient to require $\gamma \in C^m(\R;[0,1])$ instead of $\gamma \in C^\infty(\R;[0,1])$.
				
				\item \emph{The off-set $\delta$.} We  point out that the off-set $\delta$ on the top and bottom layer in our unit cell construction  is not necessary.  It is however convenient
				as by virtue of this off-set, we immediately obtain that the `corner' arising at the meeting point of $\omega_1$ and $\omega_2$ does not result in loosing the $\A$-freeness of $u$. Thus, later
				when we combine the unit cell constructions into a self-similarly refining construction, we will automatically obtain compatibility at these corners.  If we would choose $\delta = 0$, we would  need to require
				$\gamma^{(k)}(t) = 0$ for $ t \in \{0,1\}$ for $k \in\{ 1,\dots,m\}$, cf. \cite{CC15}.

                                Moreover, by this $\delta$ off-set, we can directly observe, that $\tilde{v}_k(x,0) = \tilde{v}_k(x,h) =0$ for $k\in \{1,\dots,m-1\}$.
                                
                                \item In \cref{lem:UnitCell}, we could reduce the size of $\omega$ to $[0,2^{m-1}l] \times [0,h]$, as we do not require $\int_0^{2^ml} \tilde{v}_{m-1}(t,y)dt = 0$, but it is sufficient for our construction to have $\tilde{v}_{m-1}(2^{m-1}l,y) = 0$, as we are interested in Dirichlet boundary data.
                                  
			\end{itemize}
		\end{rmk}
		
		With the unit-cell construction in hand, we proceed to the definition of a suitable cut-off layer which will be used in the top and bottom boundary regions of our self-similarly refining construction in the next section.
		
		\begin{lem}[Cut-off layer] \label{lem:UpperCutOff} Let $m \in \N$ and $\A(D),A,B, F$ be as in \cref{lem:UnitCell}.  For $0 < h \leq c l \leq 1$ for some constant $c > 0$, let
                  $\omega = (0,2^ml) \times (0,h)$ and let $E_{el}(\cdot,\cdot;\omega),E_{surf}(\cdot;\omega)$ be as in \cref{lem:UnitCell}.  Then there exists a potential $v: \omega \to \textup{Sym}(\R^2;m-1)$ such that
                  
                    \begin{align*}
                      v(0,y) & = v(2^ml,y) = 0, &\text{for } y & \in [0,h],\\
                      v(x,h) & = 0, &\text{for } x &\in [0,2^ml], \\
                      v(x,0) & = \begin{cases} -\frac{1}{2}x e_1^{\odot (m-1)} & x \in [0,\frac{l}{2}), \\  \frac{1}{2}(x-l) e_1^{\odot (m-1)} & x \in [\frac{l}{2},l],\end{cases} &\text{for } x &\in [0,l], \\
                      v(x,0) & = - v(x-2^jl,0), & \text{for } x &\in [2^jl,2^{j+1}l], j\in\{0,1,\dots,m-1\}.
			\end{align*}
                        
			Moreover, there exist $f \in BV(\omega;\{\pm \frac{1}{2}\})$ and a constant $C = C(m)>0$ such that $\chi = (A-B)f + F$, $u = \Ds v + F$ and
			\begin{align*}
				E_{el}(u,\chi;\omega) \leq C \frac{l^{2m+1}}{h^{2m-1}}, \ E_{surf}(\chi;\omega) \leq C h.
			\end{align*}
                      \end{lem}
		
		\begin{proof}
			We consider a smooth cut-off function $\phi: [0,\infty) \to [0,1]$ such that $\phi(t) = 1$ for $0 \leq t < \frac{1}{2}$ and $\phi(t) = 0$ for $t>\frac{3}{4}$, e.g., a function of a similar form as in
			\cref{rmk:UnitCellRemark}.  Let
                        
			\begin{align*}
				v& : \omega \to \textup{Sym}(\R^2;m-1), \\
				v(x,y) & = \phi(\frac{y}{h}) \tilde{f}(x) e_1^{\odot (m-1)},
			\end{align*}
			where for $x \in [0,l]$
			\begin{align*}
                          \tilde{f}(x) = \begin{cases} - \frac{1}{2}x & x \in [0,\frac{l}{2}), \\  \frac{1}{2}x - \frac{1}{2} l & x \in [ \frac{l}{2},l),
                                         \end{cases}
                                               \end{align*}
                                               $\tilde{f}(x) = -\tilde{f}(x-2^jl)$ for $x \in [2^jl,2^{j+1}l]$ for $j\in \{0,1,\dots,m-1\}$, and $\tilde{f}(x) = 0$ for $x \notin [0,2^ml]$.
Notice that $\tilde f$ is continuous in $[0,2^ml]$ since $\tilde f(0,y)=\tilde f(l,y)=0$.
                                               It is direct, that $v(x,y) = 0$ for $y > \frac{3}{4}h$ and that the properties for $y=0$ hold.
                                             Moreover, we have $|\tilde{f}(x)| \leq \frac{l}{4}$.
                                           The phase indicator is given by
                        \begin{align*}
                          \chi(x,y) = \Ds(\tilde{f}(x) e_1^{\odot (m-1)}) \allowbreak + F = f (A-B) + F,
                        \end{align*}
                        where
			$f = \tilde{f}' \in \{\pm \frac{1}{2}\}$.  As $\phi(\frac{y}{h}) \tilde{f}(x) e_1^{\odot(m-1)} = 0$ for $ y \geq \frac{3}{4}h$ or for $x \notin [0,2^ml]$ the boundary data condition is
			fulfilled.
			
			It thus remains to provide the estimate for the energy contribution in $\omega$.  For the elastic energy, we notice that with $u = \Ds v + F$
			\begin{align*}
                          |u - \chi|^2 &= \left|\Ds \Big((\phi(\frac{y}{h})-1)\tilde{f}(x) e_1^{\odot(m-1)}\Big) \right|^2 \\
                          &\leq C(m) \Big( \frac{1}{4} |1-\phi(\frac{y}{h})|^2 + |\tilde{f}(x) \frac{1}{h} \phi'(\frac{y}{h})|^2 \Big) \\
                          & \leq C(m)(1 + \frac{l^2}{h^{2}}).
			\end{align*}
			Thus,
			\begin{align*}
				E_{el}(u, \chi;\omega) \leq C(m) (1+\frac{l^2}{h^{2}}) l h \leq C(m) \frac{l^3}{h} \leq C(m) \frac{l^{2m+1}}{h^{2m-1}},
			\end{align*}
			as $h \leq c l$.  Moreover, since the interfaces are given by straight lines and since $\chi$ is bounded, we also have
			\begin{align*}
				E_{surf}(\chi;\omega) \leq C(m) h.
			\end{align*}
			Combining the two bounds hence yields the desired result.
		\end{proof}
		
		\subsection{Highest vanishing order}
		\label{sec:highest}
		
		With the unit cell construction and the cut-off function the two central ingredients of our construction are in place. We now combine these into the usual branching construction. For the construction to work, we make use of the fact that by the choice $\lambda = \frac{1}{2}$ we have that  for $v$ given in \cref{lem:UnitCell} the function $-v$ still satisfies the desired properties, most importantly $\p_1 (-\tilde{v}_0) = \tilde{u}_0 \in \{\pm \frac{1}{2}\}$.
		
		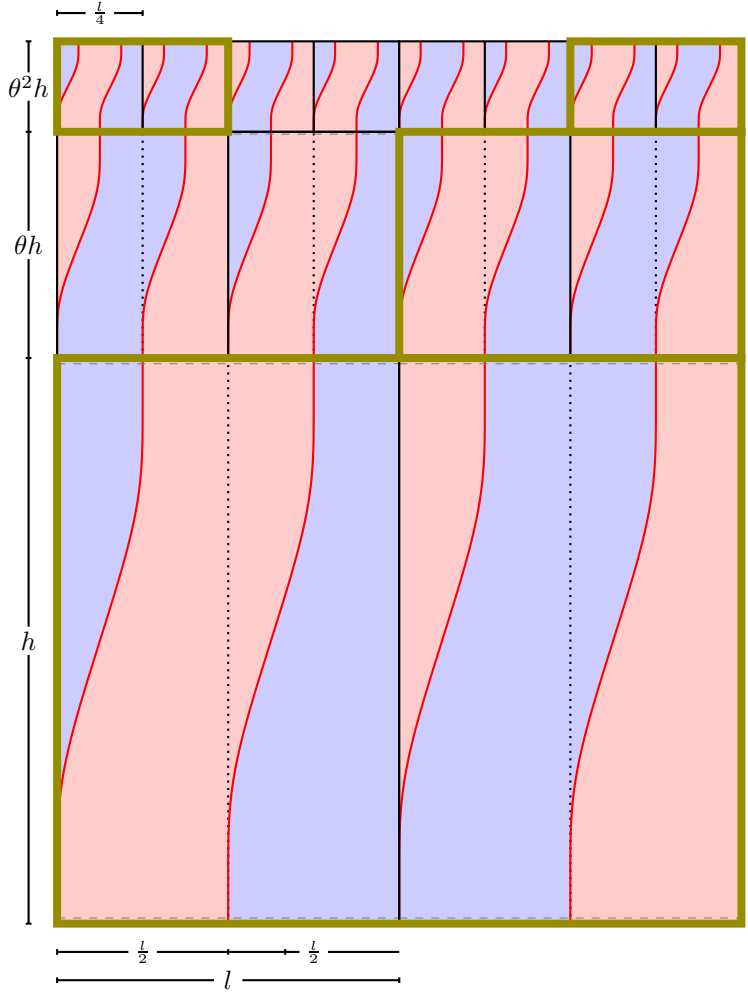
\begin{figure}
                  \centering
			\begin{tikzpicture}[thick, scale = 0.75]

				\fill[draw=none, color = blue!20] (0,0) -- (6,0) -- (6,15.6) -- (0,15.6) -- cycle;
				
                                \foreach \x in {0,4.5}
                                     \fill[draw=none, color = red!20, domain=0.21:1.4, smooth, variable=\y] (\x,14) -- plot({\x+3/(8*exp(8/(5*\y-0.08)-8/(7.92-5*\y))+8)},{\y+14}) -- (\x+0.375,15.6) -- (\x+1.125,15.6) -- plot({\x+0.75+3/(8*exp(-8/(5*\y-0.08)+8/(7.92-5*\y))+8)},{15.6-\y}) -- (\x+0.75,14) -- cycle;

                                \foreach \x in {1.5,3}
                                     \fill[draw=none, color = red!20, domain=0.21:1.58, smooth, variable=\y] (\x,14) -- plot({\x+3/(8*exp(8/(5*\y-0.08)-8/(7.92-5*\y))+8)},{\y+14}) -- (\x+0.375,15.6) -- (\x,15.6) -- cycle;

                                \foreach \x in {2.25,3.75}
                                     \fill[draw=none, color = red!20, domain=0.21:1.58, smooth, variable=\y] (\x,14) -- plot({\x+3/(8*exp(8/(5*\y-0.08)-8/(7.92-5*\y))+8)},{\y+14}) -- (\x+0.375,15.6) -- (\x+0.75,15.6) -- (\x+0.75,14) -- cycle;
                                
                                \fill[draw=none, color = red!20, domain=0.5:3.9, smooth, variable=\y] (0,10) -- plot({3/(4*exp(1/(0.25*\y-0.01)-1/(0.99-0.25*\y))+4)},{\y+10}) -- (0.75,14) -- (0,14) -- cycle;
				
                                \fill[draw=none, color = red!20, domain=0.5:3.9, smooth, variable=\y] (1.5,10) -- plot({1.5+3/(4*exp(1/(0.25*\y-0.01)-1/(0.99-0.25*\y))+4)},{\y+10}) -- (2.25,14) -- (3,14) -- (3,10) -- cycle;				
                                \fill[draw=none, color = red!20, domain=0.5:3.5, smooth, variable=\y] (3,10) -- plot({3+3/(4*exp(1/(0.25*\y-0.01)-1/(0.99-0.25*\y))+4)},{\y+10}) -- (3.75,14) -- (5.25,14) -- plot({4.5+3/(4*exp(-1/(0.25*\y-0.01)+1/(0.99-0.25*\y))+4)},{14-\y}) -- (4.5,10) -- cycle;

                                \fill[draw=none, color = red!20, domain=1.05:8.9, smooth, variable=\y] (0,0) -- plot({1.5/(exp(1/(0.1*\y-0.01)-1/(0.99-0.1*\y))+1)},{\y}) -- (1.5,10) -- (4.5,10) -- plot({3+1.5/(exp(-1/(0.1*\y-0.01)+1/(0.99-0.1*\y))+1)},{10-\y}) -- (3,0) -- cycle;
				
				\draw[dotted] (3,0) -- (3,10); \draw[dotted] (1.5,10) -- (1.5,14); \draw[dotted] (4.5,10) -- (4.5,14);

                                \foreach \x in {0,0.75,1.5,2.25,3,3.75,4.5,5.25}
                                     \draw[domain=0.21:1.58, smooth, variable=\y, red] (\x,14) -- plot({\x+3/(8*exp(8/(5*\y-0.08)-8/(7.92-5*\y))+8)},{\y+14}) -- (\x+0.375,15.6);
                               
                                \foreach \x in {0,1.5,3,4.5}
                                     \draw[domain=0.5:3.9, smooth, variable=\y, red] (\x,10) -- plot({\x+3/(4*exp(1/(0.25*\y-0.01)-1/(0.99-0.25*\y))+4)},{\y+10}) -- (\x+0.75,14);
                                
                                \draw[domain=1.05:9.85, smooth, variable=\y, red] (0,0) -- plot({1.5/(exp(1/(0.1*\y-0.01)-1/(0.99-0.1*\y))+1)},{\y}) -- (1.5,10);
                                \draw[domain=1.05:9.85, smooth, variable=\y, red] (3,0) -- plot({3+1.5/(exp(1/(0.1*\y-0.01)-1/(0.99-0.1*\y))+1)},{\y}) -- (4.5,10);
								
				\foreach \y in {0.1, 9.9, 10.04, 13.96}: \draw[dashed, gray, thin] (0,\y) -- (6,\y);
				
				\draw (0,0) -- (0,10) -- (6,10) -- (6,0) -- cycle;
                                \draw (0,10) -- (0,14) -- (6,14) -- (6,10);
                                \draw (3,10) -- (3,14);
                                \draw (0,14) -- (0,15.6) -- (6,15.6) -- (6,14) -- cycle;
                                \foreach \x in {1.5,3,4.5}
                                     \draw (\x,14) -- (\x,15.6);

                                \begin{scope}[shift={(6,0)}]
                                   	
                                  \fill[draw=none, color = red!20] (0,0) -- (6,0) -- (6,15.6) -- (0,15.6) -- cycle;
				
                                \foreach \x in {0,4.5}
                                     \fill[draw=none, color = blue!20, domain=0.21:1.4, smooth, variable=\y] (\x,14) -- plot({\x+3/(8*exp(8/(5*\y-0.08)-8/(7.92-5*\y))+8)},{\y+14}) -- (\x+0.375,15.6) -- (\x+1.125,15.6) -- plot({\x+0.75+3/(8*exp(-8/(5*\y-0.08)+8/(7.92-5*\y))+8)},{15.6-\y}) -- (\x+0.75,14) -- cycle;

                                \foreach \x in {1.5,3}
                                     \fill[draw=none, color = blue!20, domain=0.21:1.58, smooth, variable=\y] (\x,14) -- plot({\x+3/(8*exp(8/(5*\y-0.08)-8/(7.92-5*\y))+8)},{\y+14}) -- (\x+0.375,15.6) -- (\x,15.6) -- cycle;

                                \foreach \x in {2.25,3.75}
                                     \fill[draw=none, color = blue!20, domain=0.21:1.58, smooth, variable=\y] (\x,14) -- plot({\x+3/(8*exp(8/(5*\y-0.08)-8/(7.92-5*\y))+8)},{\y+14}) -- (\x+0.375,15.6) -- (\x+0.75,15.6) -- (\x+0.75,14) -- cycle;
                                
                                \fill[draw=none, color = blue!20, domain=0.5:3.9, smooth, variable=\y] (0,10) -- plot({3/(4*exp(1/(0.25*\y-0.01)-1/(0.99-0.25*\y))+4)},{\y+10}) -- (0.75,14) -- (0,14) -- cycle;
				
                                \fill[draw=none, color = blue!20, domain=0.5:3.9, smooth, variable=\y] (1.5,10) -- plot({1.5+3/(4*exp(1/(0.25*\y-0.01)-1/(0.99-0.25*\y))+4)},{\y+10}) -- (2.25,14) -- (3,14) -- (3,10) -- cycle;				
                                \fill[draw=none, color = blue!20, domain=0.5:3.5, smooth, variable=\y] (3,10) -- plot({3+3/(4*exp(1/(0.25*\y-0.01)-1/(0.99-0.25*\y))+4)},{\y+10}) -- (3.75,14) -- (5.25,14) -- plot({4.5+3/(4*exp(-1/(0.25*\y-0.01)+1/(0.99-0.25*\y))+4)},{14-\y}) -- (4.5,10) -- cycle;

                                \fill[draw=none, color = blue!20, domain=1.05:8.9, smooth, variable=\y] (0,0) -- plot({1.5/(exp(1/(0.1*\y-0.01)-1/(0.99-0.1*\y))+1)},{\y}) -- (1.5,10) -- (4.5,10) -- plot({3+1.5/(exp(-1/(0.1*\y-0.01)+1/(0.99-0.1*\y))+1)},{10-\y}) -- (3,0) -- cycle;
				
				\draw[dotted] (3,0) -- (3,10); \draw[dotted] (1.5,10) -- (1.5,14); \draw[dotted] (4.5,10) -- (4.5,14);

                                \foreach \x in {0,0.75,1.5,2.25,3,3.75,4.5,5.25}
                                     \draw[domain=0.21:1.58, smooth, variable=\y, red] (\x,14) -- plot({\x+3/(8*exp(8/(5*\y-0.08)-8/(7.92-5*\y))+8)},{\y+14}) -- (\x+0.375,15.6);
                               
                                \foreach \x in {0,1.5,3,4.5}
                                     \draw[domain=0.5:3.9, smooth, variable=\y, red] (\x,10) -- plot({\x+3/(4*exp(1/(0.25*\y-0.01)-1/(0.99-0.25*\y))+4)},{\y+10}) -- (\x+0.75,14);
                                
                                \draw[domain=1.05:9.85, smooth, variable=\y, red] (0,0) -- plot({1.5/(exp(1/(0.1*\y-0.01)-1/(0.99-0.1*\y))+1)},{\y}) -- (1.5,10);
                                \draw[domain=1.05:9.85, smooth, variable=\y, red] (3,0) -- plot({3+1.5/(exp(1/(0.1*\y-0.01)-1/(0.99-0.1*\y))+1)},{\y}) -- (4.5,10);
								
				\foreach \y in {0.1, 9.9, 10.04, 13.96}: \draw[dashed, gray, thin] (0,\y) -- (6,\y);
				
				\draw (0,0) -- (0,10) -- (6,10) -- (6,0) -- cycle;
                                \draw (0,10) -- (0,14) -- (6,14) -- (6,10);
                                \draw (3,10) -- (3,14);
                                \draw (0,14) -- (0,15.6) -- (6,15.6) -- (6,14) -- cycle;
                                \foreach \x in {1.5,3,4.5}
                                     \draw (\x,14) -- (\x,15.6);
                                  \end{scope}

				\draw (0,-1) -- (2.75,-1);
                                \draw (3.25,-1) -- (6,-1);
                                \node at (3,-1) {$l$};
                                \draw (0,-0.95) -- (0,-1.05);
                                \draw (6,-0.95) -- (6,-1.05);
				
				\draw (0,-0.5) -- (1.25,-0.5);
                                \draw (1.75,-0.5) -- (3,-0.5);
                                \node at (1.5,-0.5) {\tiny{$\frac{l}{2}$}};
                                \draw (0,-0.45) -- (0,-0.55);
                                \draw (3,-0.45) -- (3,-0.55);
				
				\draw (3,-0.5) -- (4.25,-0.5);
                                \draw (4.75,-0.5) -- (6,-0.5);
                                \node at (4.5,-0.5) {\tiny{$\frac{l}{2}$}};
                                \draw (4,-0.45) -- (4,-0.55);

                                \draw (0,16.1) -- (0.5,16.1);
                                \draw (1,16.1) -- (1.5,16.1);
                                \node at (0.75,16.1) {\tiny{$\frac{l}{4}$}};
                                \draw (0,16.05) -- (0,16.15);
                                \draw (1.5,16.05) -- (1.5,16.15);
								
				\draw (-0.5,0) -- (-0.5,4.75);
                                \draw (-0.5,5.25) -- (-0.5,10);
                                \node at (-0.5,5) {$h$};
                                \draw (-0.45,0) -- (-0.55,0);
                                \draw (-0.45,10) -- (-0.55,10);
				
				\draw (-0.5,10) -- (-0.5,11.75);
                                \draw (-0.5,12.25) -- (-0.5,14);
                                \node at (-0.5,12) {$\theta h$};
                                \draw (-0.45,14) -- (-0.55,14);

                                \draw (-0.5,14) -- (-0.5,14.55);
                                \draw (-0.5,15.05) -- (-0.5,15.6);
                                \node at (-0.5,14.8) {$\theta^2 h$};
                                \draw (-0.45,15.6) -- (-0.55,15.6);

                                \draw[line width=3, olive, opacity=0.5] (0,0) rectangle (12,10);
                                \draw[line width=3, olive, opacity=0.5] (6,10) rectangle (12,14);
				\draw[line width=3, olive, opacity=0.5] (9,14) rectangle (12,15.6);
                                \draw[line width=3, olive, opacity=0.5] (0,14) rectangle (3,15.6);
			\end{tikzpicture}
			\caption{An illustration of the branching construction used in \cref{lem:Branching}. The individual unit cell constructions from \cref{lem:UnitCell} are iteratively combined into a
				construction refining in the $e_2$ direction. In blue is the region where $\tilde{u}_0 = \tilde{A}_0$ and red corresponds to $\tilde{u}_0 = \tilde{B}_0$. The dashed horizontal lines depict the region in which we have a simple laminate. As in \cite{CC15} for $m\geq 2$ we need the curves separating the domains to be of a sufficiently high regularity (see the discussion in
				\cref{rmk:UnitCellRemark}). The unit cell and its copies, are highlighted with a green box. Moreover, for $m \geq 3$, we need to do the reflection-type argument outlined in \cref{lem:Branching} to ensure zero boundary values at the left and right.}
                              \label{fig:branching}
                            \end{figure}
		
		\begin{lem} \label{lem:Branching} Let $m \in \N$, $\Omega = (0,1)^2$ and $A,B$ be as in \cref{lem:UnitCell} and  let $F = \frac{1}{2}A + \frac{1}{2}B$.  Consider the
                  operator $\A(D)$ given in \cref{eq:Operator} and let the energy $E_{\epsilon}(u,\chi)$ for $u \in \mathcal{D}_{F}$, defined as in \cref{eq:admissible_gen}, and $\chi \in BV(\Omega;\{A,B\})$, be given by
                  \begin{align*}
                    E_\epsilon(u,\chi) = \int_\Omega |u-\chi|^2 dx + \epsilon \int_\Omega |\nabla \chi|.
                  \end{align*}
                  Then for every $N \in 2^m \N, N \geq 4$ there exist $u: \R^2 \to \textup{Sym}(\R^2;m)$,
			$f \in BV(\Omega;\{\pm \frac{1}{2}\})$ with $\A(D)u = 0$ in $\Omega$, $u = F$ outside $\Omega$, and a constant $C = C(m, \Vert \gamma \Vert_{C^m})>0$ such that for any $\epsilon >0$ it holds
			\begin{align*}
				E_{\epsilon}(u,\chi) \leq C(N^{-2m} + \epsilon N),
			\end{align*}
			where $\chi = f (A-B) + F \in BV(\Omega;\{A,B\})$.
		\end{lem}
		
		\begin{proof}
			We argue as in \cite{CC15,RRT23,RT22}.  Let $\theta \in (2^{-\frac{2m}{2m-1}},2^{-1})$ and consider the splitting $\Omega = \Omega_+ \cup \Omega_-$ with
			$\Omega_+ = [0,1] \times [\frac{1}{2},1], \Omega_- = [0,1] \times [0,\frac{1}{2}]$.  In what follows, we give the construction of $u$ and $\chi$ on $\Omega_+$.  For given $N \in 2^m\N, N \geq 4$, we
			define
			\begin{align*}
				y_j = 1 - \frac{\theta^j}{2}, \ l_j = \frac{1}{2^jN}, \ h_j = y_{j+1} - y_j, \ j \in \N \cup \{0\},
			\end{align*}
			and set $j_0 \in \N$ as the maximal $j$ satisfying $l_j < h_j$. This is possible as, due to $N \geq 4$, we have $l_0 < h_0$.  In what follows, we use \cref{lem:UnitCell} in the cells $\omega_{j,k}$, see the definition below, to achieve a
			refinement towards $y = 1$.
For the sake of clarity of exposition, we define a construction on $[0,+\infty]\times[\frac{1}{2},1]$ and then restrict it to $\Omega_+$.
			
For this, let $v^{(j)}$ be the map defined on $[0,2^ml_j] \times [0,h_j]$ according to \cref{lem:UnitCell} (with $\delta = \frac{1}{10}$) for $j \in\{0,1,\dots,j_0\}$, in particular $v^{(j)}:[0,2^ml_j] \times [0,h_j] \rightarrow \textup{Sym}(\R^2, m-1)$ is such that it satisfies the boundary conditions from \cref{lem:UnitCell}. Moreover, let $v^{(j_0+1)}$ be the map defined on $[0,2^m l_{j_0+1}] \times [0,\frac{\theta^{j_0+1}}{2}]$ according to \cref{lem:UpperCutOff}. In the following we will write $h_{j_0+1} = \frac{\theta^{j_0+1}}{2}$.
                            
We extend $v^{(j)}$ for $j\in \{0,1,\dots,j_0+1\}$ (without relabelling) onto $[0,+\infty] \times [0,h_j]$ as
\begin{align*}
v^{(j)}(x,y):=-v^{(j)}(x-2^{m+m'}l_j,y), \quad \text{for } x\in[2^{m+m'}l_j,2^{m+m'+1}l_j], \ m'\in\N.
\end{align*}
Notice that $v^{(j)}(\alpha2^{m}l_j,y)=v^{(j)}(2^ml_j,y)=0$ for every $\alpha\in\N$.
Thus $v^{(j)}(\beta,y)=0$ for every $\beta\in\N$, in particular $v^{(j)}(1,y)=0$.

			With this we define $v$ on the upper half $\Omega_+$ as follows:
			\begin{align*}
				v(x,y) = (-1)^{j} v^{(j)}(x,y-y_j), \quad \text{for } (x,y) \in [0,1]\times[y_j,y_{j+1}].
			\end{align*}

We have that $v^{(j)}(x,0) = - v^{(j-1)}(x,h_{j-1})$ for $x\in[0,1].$
Indeed, by \cref{lem:UnitCell,lem:UpperCutOff}, this is true for $x \in [0,2^ml_j]$.
Assuming that it is true for $x\in[0,2^{m+m'}l_j]$, by definition of $v^{(j)}$, and using the fact that $l_{j-1}=2l_j$, we infer that, for $x\in[2^{m+m'}l_j,2^{m+m'+1}l_j]$ we have 
\begin{align*}
v^{(j-1)}(x,h_{j-1}) 
&= - v^{(j-1)}(x-2^{m+m'-1}l_{j-1},h_{j-1}) = - v^{(j-1)}(x-2^{m+m'}l_j,h_{j-1})\\
& = v^{(j)}(x-2^{m+m'}l_j,0) = -v^{(j)}(x,0).
\end{align*}
Moreover, as already pointed out, it holds $v^{(j)}(0,y) = v^{(j)}(1,y) = 0$; 
                        thus we can extend $v(x,y) = 0$ for $x < 0$ and $x>1$. Finally, as $v^{(j_0+1)}(x,h_{j_0+1}) = 0$, we can deduce $v(x,1) = (-1)^{j_0+1} v^{(j_0+1)}(x,h_{j_0+1}) = 0$, and thus we can set $v(x,y) = 0$ for $y \geq 1$.

Now let us turn to proving the energy estimate. We denote now
\begin{align*}
\omega_{j,k} := (k2^ml_j,y_j) + [0,2^ml_j] \times [0,h_j], \  k \in \{ 0,1,\dots,N2^{j-m}-1\}, \ j \in \{ 0,1,\dots,j_0\},
\end{align*}
and
\begin{align*}
  \omega_{j_0+1,k} & := (k2^ml_{j_0+1},y_{j_0+1}) + [0,2^ml_{j_0+1}] \times [0,h_{j_0+1}].
\end{align*}
                        Setting $f = \p_1 \tilde{v}_0$, $u = \Ds v + F$, which is $\A$-free, and $\chi = f (A-B) + F$, we have
				\begin{align*}
                                  E_\epsilon (u,\chi;\omega_{j,k}) &= E_\epsilon\Big((-1)^{j} \Ds v^{(j)}+F,(-1)^j\p_1 \tilde{v}_0^{(j)} e_1^{\odot m} + F;[0,2^ml_j] \times [0,h_j]\Big) \\
			      &= E_\epsilon\Big(\Ds v^{(j)}+F,\p_1 \tilde{v}_0^{(j)} e_1^{\odot m} + F;[0,2^ml_j] \times [0,h_j]\Big) \\
                                  & \leq C(m) (\frac{l_j^{2m+1}}{h_j^{2m-1}} + \epsilon h_j),
				\end{align*}
				where we used the estimates in \cref{lem:UnitCell} and \cref{lem:UpperCutOff} for the corresponding index $j$.
			
			By symmetry we can repeat the same construction in $\Omega_-$ by replacing $e_2$ with $-e_2$. Therefore we obtain an $\A$-free map $u$ attaining the exterior data $u = F$ outside
			$\Omega$ and an associated phase indicator $\chi$.  Moreover, it holds
			\begin{align*}
				E_{\epsilon}(u,\chi) & \leq 2 \sum_{j=0}^{j_0+1} \sum_{k=0}^{N 2^{j-m}-1} \Big(E_\epsilon(u,\chi; \omega_{j,k}) + \epsilon Per(\omega_{j,k})\Big) \leq 2 C \sum_{j=0}^{j_0+1} \sum_{k=0}^{N 2^{j-m}-1} (\frac{l_j^{2m+1}}{h_j^{2m-1}} + \epsilon h_j) \\
				& = 2 C \sum_{j=0}^{j_0+1} N 2^{j-m} (\frac{l_j^{2m+1}}{h_j^{2m-1}} + \epsilon h_j) = 2 C(m) \sum_{j=0}^{j_0+1} (\frac{l_j^{2m}}{h_j^{2m-1}} + \epsilon \frac{h_j}{l_j}) \\
				& \leq 2 C(m) \sum_{j=0}^{j_0+1} \Big(N^{-2m} (2^{2m} \theta^{2m-1})^{-j} + N (2\theta)^j \Big) \leq C(m,\theta)(N^{-2m} + \epsilon N),
			\end{align*}
			which concludes the argument.
		\end{proof}
		
		\subsection{Intermediate cases and proof of the upper bounds in \cref{thm:scaling_2D_new}}
                In this section, we use the construction from \cref{sec:highest} in order to deduce an analogous construction for the case of intermediate vanishing orders for which $M:=A-B = e_1^{\odot l} \odot e_2^{\odot (m-l)}$,
		where $l,m-l \neq 0$.

		In contrast to the highest vanishing order setting, mimicking the construction from above, we now are confronted with the fact that the term involving $\chi$ is paired with two different components of the potential $v$, i.e.,
		\begin{align*}
			E_{el}(\Ds v,\chi) \sim \int_\Omega |\p_1 \tilde{v}_{m-l} + \p_2 \tilde{v}_{m-l-1} -\tilde{\chi}_{m-l}|^2 + H(\Ds v) dx,
		\end{align*}
		since now $\chi$ is nondiagonal.
		
		In order to deal with this, we make the ansatz that $\tilde{v}_{m-l-1} = 0$ or $\tilde{v}_{m-l}=0$ (depending on whether $l$ or $m-l$ is larger)  and apply the construction from \cref{lem:UnitCell} to define the other components in the energy.
		Using then the corresponding equations as in the proof of \cref{lem:UnitCell}, we set all but the last of the components for a smaller or bigger index, respectively, to zero.
		
		\begin{lem} \label{lem:Intermediate} Let $m,l \in \N$ with $0\leq l \leq m$, $\Omega = (0,1)^2$ and let $A,B$ be such that $A-B = e_1^{\odot l } \odot e_2^{\odot (m-l)}$ and $F=\frac{1}{2}A+\frac{1}{2}B$.  Consider the operator $\A(D)$ given in \cref{eq:Operator} and let the energy $E_{\epsilon}$ be given as in \cref{lem:Branching}.  We define $L := \max \{l,m-l\}$.  Then for every $N \in 2^L\N, N \geq 4$, there
			exist a deformation $u: \R^2 \to \textup{Sym}(\R^2;m)$, a phase indicator $\chi:\R^2 \to \{A,B\}$ with $\A(D)u = 0$ in $\Omega$, $u = F$ outside $\Omega$, and a constant
				$C = C(m)>0$ such that for any $\epsilon >0$ it holds
			\begin{align*}
				E_{\epsilon}(u,\chi) \leq C(N^{-2L} + \epsilon N).
			\end{align*}
		\end{lem}
		
		\begin{proof}
			The idea is to reduce the order of the tensor such that effectively we can use \cref{lem:Branching}.  As before, without loss of generality, $F = 0$ and therefore
			$A = \frac{1}{2} e_1^{\odot l} \odot e_2^{\odot(m-l)}, B = -\frac{1}{2} e_1^{\odot l} \odot e_2^{\odot (m-l)}$.  In order to reduce the order, for  $\chi = f (A-B)$, we estimate
			\begin{align*}
				E_{el}(u,\chi) \leq C(m) \left( \sum_{k=0}^{m-l-1} \int_\Omega |\tilde{u}_k|^2 dx + \int_\Omega |\tilde{u}_{m-l} - f|^2 dx + \sum_{k=m-l+1}^m \int_\Omega |\tilde{u}_k|^2 dx \right),
			\end{align*}
			where $f(x) \in \{\pm \frac{1}{2}\}$. Since the case $l=m$ (and symmetrically $l=0$) has already been considered in \cref{lem:Branching}, we only present the construction for the setting $0<l<m$. In this case, the ansatz is to set $\tilde{u}_k = 0$ for either $k \leq m-l-1$ or for $k \geq m-l+1$ (depending on whether $l>m-l$ or $m-l>l$) and to use \cref{lem:Branching} for the remaining components.   

                        \smallskip 
                                               Without loss of generality let $l \geq m-l$, else change the roles of $x$ and $y$ and adapt the coefficients accordingly.
                        We invoke \cref{lem:Branching} for $m = l \neq 0$ to define $w: \R^2 \to \textup{Sym}(\R^2;l)$, $w = 0$ outside $\Omega$, and $f: \R^2 \to \{\pm \frac{1}{2}\}$ such that
			$f (A'-B'): \R^2 \to \{A',B'\}$ with $A'= \frac{1}{2} e_1^{\odot l},B' = -\frac{1}{2} e_1^{\odot l}$. 
			We then define
			$\tilde{u}_k := (-1)^{l-m} 2^{m-l} \frac{l!k!}{m!(k-(m-l))!} \tilde{w}_{k-(m-l)}$ for $k \geq m-l$ and $\tilde{u}_k=0$ else. 
                        Moreover, we set $\chi := f (A-B) = f e_1^{\odot l} \odot e_2^{\odot (m-l)}: \R^2 \to \{A,B\}$.  Then $u$ defined by $u_{1\dots12\dots2} = \tilde{u}_k$ fulfills $u=0$ outside $\Omega$ and defines indeed an $\A$-free map as by \cref{eq:Operator2d} (denoting by $\bar{A}(D)$ the corresponding operator for $l$-th order tensors):
			\begin{align*}
				[\A(D)u]&_{1212\dots12} = \sum_{k=0}^m \Big( (-1)^k 2^{-m} \binom{m}{k} \p_1^k \p_2^{m-k} \tilde{u}_k \Big) \\
				& = \sum_{k=m-l}^m (-1)^k 2^{-m} \binom{m}{k} \p_1^k \p_2^{m-k}(-1)^{l-m} 2^{m-l} \frac{l!k!}{m!(k-(m-l))!} \tilde{w}_{k-(m-l)} \\
                                & = \sum_{k=0}^l (-1)^k 2^{-l}\frac{l!}{(m-(k+m-l))!k!}\p_1^{k+m-l} \p_2^{m-(k+m-l)} \tilde{w}_{k} \\
				& = \p_1^{m-l} \sum_{k=0}^l (-1)^k 2^{-l} \binom{l}{k} \p_1^k \p_2^{l-k} \tilde{w}_k = \p_1^{m-l} [\bar{\A}(D)w]_{12\dots12} = 0.
			\end{align*} 
		
			Furthermore, it holds
			\begin{align*}
				E_{\epsilon}(u,\chi) &\leq C \left(\int_\Omega |\tilde{w}_0 - f|^2 dx + \sum_{k=1}^{l} \int_\Omega |\tilde{w}_k|^2 dx+ \epsilon \int_\Omega |\nabla f| \right) = C \tilde{E}_\epsilon(w,fe_1^{\odot l})\\
				& \leq C (N^{-2l} + \epsilon N),
			\end{align*}
			where $\tilde{E}_\epsilon(w,fe_1^{\odot l})$ denotes the corresponding energy for $l$ tensors.
                      Since $l \geq m-l$, this yields that
                        \begin{align*}
				E_\epsilon(u,\chi) \leq C \min\{ N^{-2l} + \epsilon N, N^{-2(m-l)} + \epsilon N\} = C(N^{-2L} + \epsilon N),
			\end{align*}
                        which therefore concludes the proof.
		\end{proof}
			We combine the estimates from \cref{lem:Branching,lem:Intermediate} into the proof of the upper bound scaling result.
		
		\begin{proof}[Proof of upper bounds in \cref{thm:scaling_2D_new}]
			To show the upper bounds, we use \cref{lem:Branching}, \cref{lem:Intermediate}, and optimize in $N$, thus we choose $N \sim \epsilon^{-\frac{1}{2L+1}}$ and therefore
			\begin{align*}
				E_\epsilon(\chi;F) \leq E_\epsilon(u,\chi) \leq C (N^{-2L} + \epsilon N) \leq C\epsilon^{\frac{2L}{2L+1}}.
			\end{align*}
		\end{proof}
		
	Last but not least, we remark that the results for the higher order curl also imply a corresponding result for the higher order divergence in $d=2$.

                  \begin{cor}[Higher order divergence]
                    Let $m, k \in \N$. Let $\Omega = (0,1)^2 \subset \R^2$ and $E_\epsilon^{\mathcal{B}}$ as in \cref{eq:energy_total_gen} for the $m$-th order divergence $\mathcal{B}(D) = \di^m$ as in \cref{eq:DivergenceOp}.
                    Consider $A,B \in \R^k \otimes \textup{Sym}(\R^2;m)$ such that $A-B = v \otimes e_1^{\odot l} \odot e_2^{\odot (m-l)}$ for $l \in \{0,1,\dots,m\}, v\in \R^k$ and $F = \frac{1}{2}A+\frac{1}{2}B$.
                    Then there exist $\chi \in BV(\Omega;\{A,B\})$ and $C = C(m) > 1$ such that for $L := m - \min\{l,m-l\} = \max\{l,m-l\}$  
                    \begin{align*}
                      E_\epsilon(\chi;F) \leq C \epsilon^{\frac{2L}{2L+1}}.
                    \end{align*}
                  \end{cor} 
                  \begin{proof}
                    First we notice that, without loss of generality we can assume $k=1$ by working componentwise and moreover, using the notation of \cref{rmk:Notation-2D}, $\di^m u = \sum_{j=0}^m \p_1^j \p_2^{m-j} \tilde{u}_{m-j}$.
                    The idea is to use the $\A$-free setting for $\A$ as in \cref{lem:Intermediate} and then transform $u',\chi'$ such that we are in the divergence-free setting.
                    To simplify the notation, set
                    \begin{align*}
                      \alpha(m,j) = (-1)^j 2^{-m} \binom{m}{j} \neq 0,
                    \end{align*}
                    such that $[\A(D)u']_{12\dots12} = \sum_{j=0}^m \alpha(m,j) \p_1^j \p_2^{m-j} \tilde{u}'_j$. 
                    By an application of \cref{lem:Intermediate}, for $N \in 2^L\N$, $A',B' \in \textup{Sym}(\R^2;m)$ and $F' = \frac{1}{2}A' + \frac{1}{2}B'$, where $A'$ and $B'$ are defined componentwise by $\tilde{A}'_j = \alpha(m,j)^{-1} \tilde{A}_{m-j}$, $\tilde{B}'_j = \alpha(m,j)^{-1} \tilde{B}_{m-j}$, and for $F' = \frac{1}{2}A' + \frac{1}{2}B'$ there exists $u':\R^2 \to \textup{Sym}(\R^2;m)$, $\chi':\R^2 \to \{A',B'\}$ such that $u' = F'$ outside $\Omega$ and $\A(D)u' = 0$.
                    We can apply \cref{lem:Intermediate} after a rescaling, as for the above definition of $A',B'$ we have $A'-B' = \alpha(m,l)^{-1} e_1^{\odot (m-l)} \odot e_2^{\odot l}$. This can be seen by considering $\widetilde{(A'-B')}_j = \tilde{A}'_j - \tilde{B}'_j= \alpha(m,j)^{-1} \cdot\widetilde{(A-B)}_{m-j}$ and using that $A-B = e_1^{\odot l} \odot e_2^{\odot (m-l)}$.
                    Moreover it holds 
                    \begin{align*}
                      E_\epsilon^{\A}(u',\chi') \leq C (N^{-2L} + \epsilon N)
                    \end{align*}
                    with a constant $C > 0$ only depending on $m$.
                    Setting now
                    \begin{align*}
                      \tilde{u}_{m-j} := \alpha(m,j)  \tilde{u}'_j, \ \tilde{\chi}_{m-j} := \alpha(m,j)  \tilde{\chi}'_j,
                    \end{align*}
                    we observe that
                    \begin{align*}
                      \di^m u = \sum_{j=0}^m \p_1^j \p_2^{m-j} \tilde{u}_{m-j} = [\A(D)u']_{12 \dots 12} = 0
                    \end{align*}
                    and that outside $\Omega$ we have
                    \begin{align*}
                      \tilde{u}_{m-j} = \alpha(m,j)  \tilde{u}'_j = \alpha(m,j) \tilde{F}'_j = \alpha(m,j) \Big( \frac{1}{2} \tilde{A}'_j + \frac{1}{2} \tilde{B}'_j \Big) =  \tilde{F}_{m-j}.
                    \end{align*}
                    Thus $u$ is admissible and $\tilde{\chi}_{m-j} = \alpha(m,j) \tilde{\chi}'_j \in \{\tilde{A}_{m-j},\tilde{B}_{m-j}\}$.
                    To bound the energy we notice that
                    \begin{align*}
                      |u-\chi|^2 = \sum_{j=0}^m \binom{m}{j} |\tilde{u}_j - \tilde{\chi}_j|^2 \leq C(m) |u'-\chi'|^2, \quad |\nabla \chi| \leq C(m) |\nabla \chi'|,
                    \end{align*}
                    and therefore $E_\epsilon^{\mathcal{B}}(u,\chi) \leq C(m) E_\epsilon^{\A}(u',\chi') \leq C(m) (N^{-2L} + \epsilon N)$.
                    Choosing $N \sim \epsilon^{-\frac{1}{2L+1}}$ concludes the proof.
                  \end{proof}

                  \section*{Acknowledgments}
                  A.R. and A.T. acknowledge funding by the Deu\-tsche For\-schungs\-ge\-mein\-schaft (DFG, German Research Foundation) through SPP 2256, pro\-ject ID 441068247. 
                  All authors were partially supported by the Hausdorff Institute for Mathematics at the University of Bonn which is funded by the Deutsche Forschungsgemeinschaft (DFG, German Research Foundation) under Germany's Excellence Strategy – EXC-2047/1 – 390685813, as part of the Trimester Program on Mathematics for Complex Materials. A.R. and C.T. are supported by the Hausdorff Center for Mathematics which is funded by the Deutsche Forschungsgemeinschaft (DFG, German Research Foundation) under Germany's Excellence Strategy – EXC-2047/1.
              
		\bibliographystyle{alpha} \bibliography{references.bib}
		
	\end{document}